\documentclass[11pt, leqno, twoside]{article}
\usepackage{}

\usepackage{amsfonts}

\usepackage{amssymb}
\usepackage{amsmath}
\usepackage{amsthm}
\usepackage{xcolor}
\usepackage{mathrsfs}
\usepackage{txfonts}

\usepackage{indentfirst}

\allowdisplaybreaks

\def\zz{{\mathbb Z}}

\def\fz{\infty}
\def\az{\alpha}
\def\bz{\beta}

\def\lz{\lambda}
\def\Lz{\Lambda}
\def\sz{\sigma}
\def\oz{\omega}
\def\Oz{\Omega}
\def\pz{{\,\oplus\,}}
\def\mz{{\,\ominus\,}}
\def\Pz{{\mathbb{P}}}
\def\nn{{\mathbb{N}}}

\def\lf{\left}
\def\r{\right}

\def\noz{\nonumber}

\def\supp{\mathop\mathrm{\,supp\,}}

\newtheorem{theorem}{Theorem}[section]
\newtheorem{lemma}[theorem]{Lemma}

\theoremstyle{definition}
\newtheorem{remark}[theorem]{Remark}
\newtheorem{definition}[theorem]{Definition}
\renewcommand{\appendix}{\par
   \setcounter{section}{0}%
   \setcounter{subsection}{0}%
   \setcounter{subsubsection}{0}%
   \gdef\thesection{\@Alph\c@section}%
   \gdef\thesubsection{\@Alph\c@section.\@arabic\c@subsection}%
   \gdef\theHsection{\@Alph\c@section.}%
   \gdef\theHsubsection{\@Alph\c@section.\@arabic\c@subsection}%
   \csname appendixmore\endcsname
 }

\numberwithin{equation}{section}

\begin{document}

\arraycolsep=1pt

\title{\bf\Large A Characterization of Wavelet Sets
on Vilenkin Groups with Its Application to Construction of MRA Wavelets\footnotetext {\hspace{-0.35cm}
2020 {\it Mathematics Subject Classification}. Primary 42C40;
Secondary 43A70, 42C10, 22B99.
\endgraf {\it Key words and phrases}. Vilenkin group, wavelet set, multiresolution analysis.
\endgraf
This project is supported by the National
Natural Science Foundation of China (Grant No.~12371102).}}
\author{Jun Liu and Chi Zhang\footnote{Corresponding author, E-mail: zclqq32@cumt.edu.cn}}
\date{}
\maketitle

\vspace{-0.9cm}

\begin{center}
\begin{minipage}{13cm}
{\small {\bf Abstract}\quad
Let $G$ be a Vilenkin group. In 2008, Y. A. Farkov constructed wavelets on $G$ via the
multiresolution analysis method. In this article, a characterization of wavelet sets on $G$ is established,
which provides another method for the construction of wavelets. As an application, the relation between multiresolution analyses and wavelets determined from wavelet sets is also presented.
To some extent, these results positively answer a question mentioned by P. Mahapatra and D. Singh in
[Bull. Sci. Math. 167 (2021), Paper No. 102945, 20 pp].}
\end{minipage}
\end{center}

\section{Introduction}\label{s1}

As a natural generalization of the Cantor dyadic group, the Vilenkin group was first introduced by Vilenkin
in \cite{Vilen47}, which is also a large class of locally compact abelian groups. Actually, Fine \cite{Fine}
and Vilenkin \cite{Vilen47} independently determined that Walsh system is the group of characters of the
Cantor dyadic group. Nowadays, these locally compact abelian groups and their analogue have proved very
useful not only in the development of group theory, but also in many other branches of mathematics,
such as functional analysis (see, for instance, \cite{gi04,jm23,ls09,m17}) and harmonic analysis
(see, for instance, \cite{amw12,bb11,Far08,fg23,Ho21,ly93,ly98,oq89,qy02}).

In particular, Lang \cite{Lang96} studied orthogonal wavelets with compact support on the Cantor dyadic
group, in which, the author used multiresolution analysis (MRA) to construct the desired wavelets and
determined relations for scaling and wavelet filters similar to the case of $L^2(\mathbb{R})$.
With the help of the
theory of Calder\'{o}n-Zygmund integral operators, Lang \cite{Lang98} further showed that, if wavelets
on the Cantor dyadic group satisfy a Lipschitz-type regularity condition, then the convergence of the
wavelet series in $L^r$ is unconditional (that is regardless of the ordering of summation index set)
for any $r>1$.
These results were extended to the Vilenkin group by Farkov in \cite{Far07,Far08}. To be precise,
in \cite{Far08}, the author gave necessary and sufficient conditions, in terms of the modified Cohen's
condition and blocked sets, for refinable functions to generate an MRA in
the $L^2$-space on Vilenkin groups. The Strang-fix condition, the partition of unity property,
the linear independence and the stability of the corresponding refinable (or scaling) functions are
also investigated in \cite{Far08}. For more progresses about wavelet theory on Vilenkin groups, we refer
the reader to \cite{b19,fls15,fs22,l14,l14*,lb15,lbk15,mss23}.

On the other hand, Dai and Larson \cite{dl98} introduced and studied the wavelet set in $\mathbb{R}$,
which was defined as follows: A
measurable subset of $\mathbb{R}$ is called a wavelet set if its characteristic function
is equal to the Fourier transform (or, modulus of Fourier transform) of a wavelet.
Such wavelets were also called $s$-elementary wavelets. These were later extended to higher dimensional
Euclidean space $\mathbb{R}^n$ in \cite{dls97,dls98}.
In real case, wavelet sets has become a powerful tool for the construction of
MRA as well as non-MRA wavelets. Due to the celebrated work \cite{dl98,dls97,dls98} of Dai et al.,
there has been an increasing interest in investigating geometric and topological properties of wavelet
sets and associated wavelets in various underlying spaces; see, for instance,
\cite{bb11,br20,ms21,mss23} and their references. Recently, Mahapatra and Singh \cite{ms21}
considered the wavelet sets on Cantor dyadic groups and their association with MRA; moreover,
Mahapatra and Singh in \cite[p.\,2]{ms21} stated that
``To the best of our knowledge, results related to scaling sets
are not available in case of locally compact abelian groups. Therefore the construction
of scaling sets may be extended from Cantor dyadic group to locally compact abelian
groups".

Motivated by the aforementioned work of \cite{Far08} and \cite{ms21}, in this article, we
study the wavelet sets and their association with MRA on Vilenkin groups. Being more precise,
in Section \ref{s2}, we first give the algebraic and topological structure of the Vilenkin group $G$
along with some existing results on wavelets, and then prove several basic results related to wavelets;
see Lemmas \ref{2l2} through \ref{2l4} below. Section \ref{s3} is devoted to establishing a
characterization of wavelet sets on $G$ (see Theorem \ref{t1} below), whose proof
borrows some ideas from the proof of \cite[Theorem 4]{ms21}. We should point out that the obtained
characterization also can be found in \cite[Theorem 2.4]{mss23}. However, the authors of \cite{mss23}
did not give its proof. For the sake of completeness and to support our claim in Theorem \ref{t2} of
Section \ref{s4}, detailed proofs of Theorem \ref{t1} are given in the present article.
In Section \ref{s4}, using Theorem \ref{t1}, we establish the relation between
MRA and wavelets determined from wavelet sets on $G$. This provides another method, which is different
from that used in \cite{Far08}, for the construction of wavelets.
Observe that Vilenkin groups are locally compact abelian groups. In this sense, we confirm the
aforementioned conjecture proposed by Mahapatra and Singh \cite{ms21} to some extent.

Throughout this article, let $\zz$ be \emph{the set of integers}, $\nn:=\{1,2,\ldots\}$, and $\zz_+:=\{0\}\cup\nn$.
Let $p\in[2,\fz)\cap\zz$ and $\mathbb{P}:=\{1,2,\ldots,p-1\}$.
For any set $E$, we denote by $\mathbf{1}_E$ its \emph{characteristic function}.

\section{Preliminaries}\label{s2}

In this section, we first recall some notations and definitions from \cite{Far08,Vilen47}
which will be used throughout this article and then show several basic results related to
wavelets on Vilenkin groups.

We begin with the notion of Vilenkin groups. Let $G$ be the set of all sequences
$$x=(x_j)_{j\in\zz}=(\ldots,0,0,x_k,x_{k+1},\ldots)$$
satisfying that $x_j\in\{0,1,\ldots,p-1\}$ for any $j\in\zz$ and $x_j=0$ for every
$j\in(-\fz,k_x)\cap\zz$, where $k_x$ is some integer depending on $x$. In what follows,
we always use $x$ or $(x_j)$ to denote $(x_j)_{j\in\zz}$ for convenience.
The group operation on $G$, denoted by $\oplus$, is coordinatewise addition modulo $p$
defined by setting
$$(z_j)=(x_j)\pz(y_j)\Longleftrightarrow{\rm for}\ {\rm any}\ j\in\zz,\
z_j=x_j+y_j({\rm mod}\,p).$$
In addition, denote by $\mz$ the inverse operation of $\pz$.
An automorphism $\rho:G\to G$ is called a dilation if it satisfies that,
for any $x\in G$ and $j\in\zz$, $(\rho(x))_j=x_{j+1}$. Its inverse $\sz$ is defined as,
for any $x\in G$ and $j\in\zz$, $(\sz(x))_j=x_{j-1}$. Moreover, for any $m\in\zz$,
let $U_m:=\rho^m(U)$, where
\begin{align}\label{2e1}
U:=\lf\{(x)_j\in G:\ x_j=0\ {\rm for}\ j\in (-\fz,0]\cap\zz\r\}.
\end{align}
Then the set $\mathcal{U}:=\{U_m:\ m\in\zz\}$ forms a local base at the identity element
$\theta:=(\ldots,0,0,\ldots)$ and $\mathcal{B}_{\mathcal{U}}:=\{x+U_m:\ x\in G,\ m\in\zz\}$
is the base for the topology on $G$. With respect to
the topology and the group operation given above, $G$ forms a topological group which
is called a Vilenkin group.

The dual group $G^*$ of $G$ is topologically isomorphic to $G$, which consists of all sequences
of the form
$$\oz=(\oz_j)_{j\in\zz}=\lf(\ldots,0,0,\oz_k,\oz_{k+1},\ldots\r),$$
where $\oz_j\in\{0,1,\ldots,p-1\}$ for any $j\in\zz$ and $\oz_j=0$ for every
$j\in(-\fz,k_\oz)\cap\zz$ with some integer $k_\oz$ depending on $\oz$.
The group operation and topology on $G^*$ are defined as above for $G$. And also,
the dilation $\rho$ and its inverse $\sz$ on $G^*$ have the same definitions as on the group $G$.
Furthermore, each character on $G$ can be defined by setting, for any $x\in G$,
$$\chi(x,\oz):=\exp\lf(\frac{2\pi i}{p}\sum_{j\in\zz}x_{-j}\oz_{j-1}\r)$$
with some $\oz\in G^*$.

For any $q\in[1,\fz)$, the Lebesgue space $L^q(G)$ (resp. $L^q(G^*)$) is defined via the
Haar measure $\mu$ (resp. $\mu^*$) on
Borel subsets of $G$ (resp. $G^*$) normalized by $\mu(U)=1$ (resp. $\mu^*(U^*)=1$),
where $U^*$ is the subgroup of $G^*$ defined similarly as in \eqref{2e1}. To be precise,
the Lebesgue space $L^q(G)$ is defined to be the set of all measurable functions $f$ on $G$
such that
$$\|f\|_{L^p(G)}:=\lf[\int_G|f(x)|^p\,d\mu(x)\r]^{1/p}<\fz,$$
and $L^q(G^*)$ is defined as $L^q(G)$ with $G$ replaced by $G^*$. Denote by $(\cdot,\cdot)$ and
$\|\cdot\|$, respectively, the inner product and the norm of $L^2(G)$ or $L^2(G^*)$.
For any function $f\in L^1(G)\cap L^2(G)$, the Fourier transform $\widehat{f}$ is defined by setting,
for any $\oz\in G^*$,
$$\widehat{f}(\oz):=\int_{G}f(x)\overline{\chi(x,\oz)}\,d\mu(x),$$
which belongs to the space $L^2(G)$. Moreover,
for any $f\in L^1(G^*)\cap L^2(G^*)$, the inverse Fourier transform $f^\vee$ is defined by setting,
for any $x\in G$,
$$f^\vee(x):=\int_{G^*}f(\oz)\chi(x,\oz)\,d\mu^*(\oz).$$
Observe that the Fourier operator and the inverse Fourier operator
can be extended in the standard way, respectively, to the whole space $L^2(G)$ and $L^2(G^*)$.

Let
$$\Lz:=\lf\{(x)_j\in G:\ x_j=0\ {\rm for}\ j\in (0,\fz)\cap\zz\r\}$$
and, for any $j\in\zz$, $\Lz_j:=\rho^j(\Lz)$. Then $\Lz$ is a countable and closed subgroup of $G$
and, for any $j\in\zz$, $\Lz_{j+1}\subset\Lz_j$. In addition, the subgroup $\Lz$ of $G^*$
is defined similarly, as for $G$. Note that the bijective map $\lz:\ \Lz\to\zz_+$, defined by
$$\lz(x):=\sum_{j\in\zz}x_jp^{-j},\quad\forall\,x=(x_j)\in\Lz,$$
identifies elements of $\Lz$ and $\zz_+$. Therefore, we will use the same notation for elements
of $\Lz$ and $\zz_+$ unless there is confusion.

We next recall the concept of the multiresolution analysis (see, for instance, \cite{Far08}).

\begin{definition}\label{2d1}
Let $G$ be a Vilenkin group.
An increasing sequence of closed subspaces in $L^2(G)$,
denoted by $\{V_j\}_{j\in\zz}$, is called a {\it multiresolution analysis} (for short, MRA)
\emph{of $L^2(G)$} if it has the properties as follows:

\begin{enumerate}
\item [{\rm(i)}]
$\bigcup_{j\in\zz} V_j$ is dense in $L^2(G)$ and
$\bigcap_{j\in\zz} V_j=\{\Theta\}$,
where $\Theta$ denotes the zero element of $L^2(G)$;

\item [{\rm(ii)}]
for each $j\in\zz$ and $f\in L^2(G)$, $f(\cdot)\in V_j$
if and only if $\rho f(\cdot):=f(\rho\cdot)\in V_{j+1}$;

\item [{\rm(iii)}]
for any $f\in L^2(G)$, $f(\cdot)\in V_0$ if and only if $f(\cdot\mz n)\in V_{0}$
for any $n\in\Lz$;

\item [{\rm(iv)}]
there exists an $L^2(G)$-function $\phi$ (called a {\it scaling function}) such that
$\{\phi(\cdot\mz n)\}_{n\in\Lz}$ forms a \emph{Riesz basis} of $V_0$, that is,
for any sequence $\{\bz_n\}_{n\in\Lz}\in \ell^2$,
$$C\lf(\sum_{n\in\Lz}|\bz_n|^2\r)^{1/2}
\le\lf\|\sum_{n\in\Lz} \bz_n \phi(\cdot\mz n)\r\|
\le \widetilde{C}\lf(\sum_{n\in\Lz}|\bz_n|^2\r)^{1/2},$$
where $C$ and $\widetilde{C}$ are two positive constants independent of $\{\bz_n\}_{n\in\Lz}$.
\end{enumerate}
\end{definition}
Let $\{V_j\}_{j\in\zz}$ be the same as in Definition \ref{2d1} and, for any $j\in\zz$,
$W_j$ the {\it orthogonal complement} of $V_{j}$ in $V_{j+1}$. Then we have
\begin{align*}
V_{j+1}=\bigoplus_{i\in(-\fz,j]\cap\zz} W_i
\ \text{and}\
L^2(G)=\bigoplus_{i\in\zz} W_i.
\end{align*}
Hereinafter, for any $f\in L^2(G)$, $j\in\zz$, $n\in\Lz$ and $x\in G$, we always let
$$f_{j,n}(x):=p^{j/2}f\lf(\rho^j(x)\mz n\r).$$

The following lemma comes from \cite{Far08}.

\begin{lemma}\label{2l1}
There exists a sequence $\{V_j\}_{j\in\zz}$, which are the same as in
Definition \ref{2d1}, and families of scaling functions $\{\phi_{j,n}\}_{j\in\zz,n\in\Lz}$
and wavelets $\{\psi_{j,n}^{(u)}\}_{j\in\zz,n\in\Lz,u\in\Pz}$ satisfying that,
for any $j\in\zz$, the set $\{\phi_{j,n}\}_{n\in\Lz}$ forms an orthonormal basis
of $V_j$ and the set $\{\psi_{j,n}^{(u)}\}_{n\in\Lz,u\in\Pz}$ an orthonormal basis of $W_j$.
Furthermore, the set $\{\psi_{j,n}^{(u)}\}_{j\in\zz,n\in\Lz,u\in\Pz}$ forms
an orthonormal basis of $L^2(G)$.
\end{lemma}

First, we have the succeeding conclusion, which implies that, if $\phi$ is a scaling function
of an MRA, then $\sum_{n\in\Lz}|\widehat{\phi}(\oz\pz n)|^2=1$ for almost every $\oz\in G^*$.

\begin{lemma}\label{2l2}
Let $f\in L^2(G)$. Then $\{f(\cdot\pz n)\}_{n\in\Lz}$ forms an orthonormal set if and only if
\begin{align}\label{2e2}
\sum_{n\in\Lz}\lf|\widehat{f}(\oz\pz n)\r|^2=1\ for\ almost\ every\ \oz\in G^*.
\end{align}
\end{lemma}

\begin{proof}
We first show the necessity of the present lemma. To this end, let $\{f(\cdot\pz n)\}_{n\in\Lz}$
be an orthonormal set in $L^2(G)$. Note that, for any $x\in G$,
$$f(x)=\lf(\widehat{f}\,\r)^\vee(x)=\int_{G^*}\widehat{f}(\oz)\chi(x,\oz)\,d\mu^*(\oz).$$
Then, for any $k\in\Lz$, we have
\begin{align}\label{2e3}
\lf(f(\cdot),f(\cdot\pz k)\r)
&=\int_G\overline{f(x)}f(x\pz k)\,d\mu(x)\\
&=\int_G\overline{f(x)}\int_{G^*}\widehat{f}(\oz)\chi(x\pz k,\oz)\,d\mu^*(\oz)d\mu(x)\noz\\
&=\int_{G^*}\widehat{f}(\oz)\chi(k,\oz)\int_G\overline{f(x)}\chi(x,\oz)\,d\mu(x)d\mu^*(\oz)\noz\\
&=\int_{G^*}\widehat{f}(\oz)\chi(k,\oz)
\overline{\int_Gf(x)\overline{\chi(x,\oz)}\,d\mu(x)}d\mu^*(\oz)\noz\\
&=\int_{G^*}\lf|\widehat{f}(\oz)\r|^2\chi(k,\oz)d\mu^*(\oz)\noz\\
&=\int_{U^*}\sum_{n\in\Lz}\lf|\widehat{f}(\oz\pz n)\r|^2\chi(k,\oz)d\mu^*(\oz).\noz
\end{align}
Thus, the Fourier coefficients of the $\Lz$-periodic function
$\sum_{n\in\Lz}|\widehat{f}(\oz\pz n)|^2$, namely,
$$c_k=\int_{U^*}\sum_{n\in\Lz}\lf|\widehat{f}(\oz\pz n)\r|^2\chi(k,\oz)d\mu^*(\oz)$$
equals to $0$ for any $k\in\Lz\backslash\{\theta\}$ and $c_{\theta}=1$. This further implies
that \eqref{2e2} holds true.

The sufficiency of this lemma can be verified by an argument
similar as above; the details are omitted, which completes the proof of Lemma \ref{2l2}.
\end{proof}

Moreover, we give the following useful scaling and wavelet equations.

\begin{lemma}\label{2l3}
Let $\{V_j\}_{j\in\zz}$ be an {\rm MRA} as in Definition \ref{2d1}. Assume that $\phi$ is a scaling
function and $\{\psi^{(u)}\}_{u\in\Pz}$ a sequence of wavelets associated with $\{V_j\}_{j\in\zz}$.
Then there exist $\Lz$-periodic functions $m_0$ and $\{m_1^{(u)}\}_{u\in\Pz}$, which are, respectively,
called scaling and wavelet filters such that,
\begin{enumerate}
\item [{\rm(i)}] for any $\oz\in G^*$,
\begin{align*}
\widehat{\phi}(\oz)=m_0(\sz(\oz))\widehat{\phi}(\sz(\oz));
\end{align*}
\item [{\rm(ii)}] for any $u\in\Pz$ and $\oz\in G^*$,
\begin{align*}
\widehat{\psi^{(u)}}(\oz)=m_1^{(u)}(\sz(\oz))\widehat{\phi}(\sz(\oz)).
\end{align*}
\end{enumerate}
\end{lemma}

\begin{proof}
We first prove (i). For this purpose, by (ii) and (iv) of Definition \ref{2d1}, we find that,
for any $j\in\zz$ and $f\in L^2(G)$, $f\in V_0$ if and only if $\rho^jf\in V_j$, and
$\{\phi_{j,n}\}_{n\in\Lz}$ forms an orthonormal basis of $V_j$. This, combined with the fact
that $\phi\in V_0\subset V_1$, implies that, for any $x\in G$,
\begin{align*}
\phi(x)=\sum_{n\in\Lz}a_n\phi(\rho(x)\pz n),
\end{align*}
where, for any $n\in\Lz$,
\begin{align*}
a_n:=p\int_G\phi(x)\overline{\phi(\rho(x)\pz n)}\,d\mu(x).
\end{align*}
From this, we infer that, for any $\oz\in G^*$,
\begin{align}\label{2e6}
\widehat{\phi}(\oz)&=\int_G\phi(x)\overline{\chi(x,\oz)}\,d\mu(x)\\
&=\int_G\sum_{n\in\Lz}a_n\phi(\rho(x)\pz n)\overline{\chi(x,\oz)}\,d\mu(x)\noz\\
&=\int_G\sum_{n\in\Lz}a_n\phi(z)\overline{\chi(\sz(z\pz n),\oz)}\,d\mu(\sz(z))\noz\\
&=\int_G\sum_{n\in\Lz}\frac{a_n}p\phi(z)\overline{\chi(\sz(z),\oz)}\,\overline{\chi(\sz(n),\oz)}\,d\mu(z)\noz\\
&=\sum_{n\in\Lz}\frac{a_n}p\overline{\chi(\sz(n),\oz)}\int_G\phi(z)\overline{\chi(\sz(z),\oz)}\,d\mu(z)\noz\\
&=m_0(\sz(\oz))\widehat{\phi}(\sz(\oz)),\noz
\end{align}
where
\begin{align*}
m_0(\sz(\oz)):=\sum_{n\in\Lz}\frac{a_n}p\overline{\chi(\sz(n),\oz)}
\end{align*}
is a $\Lz$-periodic function. This finishes the proof of (i).

Next, we show (ii). To do this, since $\{\psi^{(u)}\}_{u\in\Pz}\subset W_0\subset V_1$,
it follows that, for any $u\in\Pz$ and $x\in G$,
\begin{align*}
\psi^{(u)}(x)=\sum_{n\in\Lz}b^{(u)}_n\phi(\rho(x)\pz n),
\end{align*}
where, for any $u\in\Pz$ and $n\in\Lz$,
\begin{align*}
b^{(u)}_n:=p\int_G\psi^{(u)}(x)\overline{\phi(\rho(x)\pz n)}\,d\mu(x).
\end{align*}
By this and a calculation similar to \eqref{2e6}, we easily obtain that,
for any $u\in\Pz$ and $\oz\in G^*$,
\begin{align*}
\widehat{\psi^{(u)}}(\oz)=m^{(u)}_1(\sz(\oz))\widehat{\phi}(\sz(\oz)),
\end{align*}
where
\begin{align*}
m^{(u)}_1(\sz(\oz)):=\sum_{n\in\Lz}\frac{b^{(u)}_n}p\overline{\chi(\sz(n),\oz)}
\end{align*}
is a $\Lz$-periodic function for each $u\in\Pz$, which completes the proof of (ii)
and hence of Lemma \ref{2l3}.
\end{proof}

Using Lemmas \ref{2l1} through \ref{2l3}, we obtain several identical relations for scaling and
wavelet filters as follows.

\begin{lemma}\label{2l4}
Let $m_0$ and $\{m^{(u)}_1\}_{u\in\Pz}$ be, respectively, the scaling filter and wavelet filters
as in Lemma \ref{2l3}. Then, for any $u,\tau\in\Pz$ with $u\neq\tau$ and almost every $\oz\in G^*$,
it holds true that
\begin{align}\label{2e7}
\lf|m_0(\oz)\r|^2+\sum_{x\in\Pz}\lf|m_0\lf(\oz\pz0.x\r)\r|^2=1,
\end{align}
\begin{align}\label{2e8}
\lf|m^{(u)}_1(\oz)\r|^2+\sum_{x\in\Pz}\lf|m^{(u)}_1\lf(\oz\pz0.x\r)\r|^2=1,
\end{align}
\begin{align}\label{2e9}
m_0(\oz)\overline{m^{(u)}_1(\oz)}+\sum_{x\in\Pz}m_0(\oz\pz0.x)\overline{m^{(u)}_1(\oz\pz0.x)}=0
\end{align}
and
\begin{align}\label{2e9'}
m^{(u)}_1(\oz)\overline{m^{(\tau)}_1(\oz)}+
\sum_{x\in\Pz}m^{(u)}_1(\oz\pz0.x)\overline{m^{(\tau)}_1(\oz\pz0.x)}=0,
\end{align}
where $0.x:=(\nu_j)_{j\in\zz}$ is the element of $G^*$ satisfying that $\nu_1=x\in\Pz$
and $\nu_j=0$ for any $j\in\zz\setminus\{1\}$.
\end{lemma}

\begin{proof}
We first show \eqref{2e7}. To this end, by Lemma \ref{2l1},
we know that $\{\phi(\cdot\pz n)\}_{n\in\Lz}$ is an orthonormal set in $L^2(G)$. This, together with
Lemma \ref{2l2}, implies that
\begin{align*}
\sum_{n\in\Lz}\lf|\widehat{\phi}(\oz\pz n)\r|^2=1\ {\rm for\ almost\ every}\ \oz\in G^*.
\end{align*}
From this, Lemma \ref{2l3}, the definition of $\{\Lz_j\}_{j\in\zz}$ and the fact that
$m_0$ is a $\Lz$-periodic function, we deduce that, for almost every $\oz\in G^*$,
\begin{align}\label{2e10}
1&=\sum_{n\in\Lz}\lf|\widehat{\phi}(\rho(\oz)\pz n)\r|^2\\
&=\sum_{n\in\Lz}\lf|m_0(\sz(\rho(\oz)\pz n))\r|^2\lf|\widehat{\phi}(\sz(\rho(\oz)\pz n))\r|^2\noz\\
&=\sum_{n\in\Lz}\lf|m_0(\oz\pz\sz(n))\r|^2\lf|\widehat{\phi}(\oz\pz\sz(n))\r|^2\noz\\
&=\sum_{n\in\Lz_1}\lf|m_0(\oz\pz\sz(n))\r|^2\lf|\widehat{\phi}(\oz\pz\sz(n))\r|^2\noz\\
&\quad+\sum_{x\in\Pz}\sum_{n\in\Lz_1}
\lf|m_0(\oz\pz\sz(n)\pz0.x)\r|^2\lf|\widehat{\phi}(\oz\pz\sz(n)\pz0.x)\r|^2\noz\\
&=\lf|m_0(\oz)\r|^2\sum_{n\in\Lz_1}\lf|\widehat{\phi}(\oz\pz\sz(n))\r|^2
+\sum_{x\in\Pz}\lf|m_0(\oz\pz0.x)\r|^2\sum_{n\in\Lz_1}\lf|\widehat{\phi}(\oz\pz\sz(n)\pz0.x)\r|^2\noz\\
&=\lf|m_0(\oz)\r|^2\sum_{n\in\Lz}\lf|\widehat{\phi}(\oz\pz n)\r|^2
+\sum_{x\in\Pz}\lf|m_0(\oz\pz0.x)\r|^2\sum_{n\in\Lz}\lf|\widehat{\phi}(\oz\pz n\pz0.x)\r|^2\noz\\
&=\lf|m_0(\oz)\r|^2+\sum_{x\in\Pz}\lf|m_0(\oz\pz0.x)\r|^2.\noz
\end{align}
Thus, \eqref{2e7} holds true.

To verify \eqref{2e8}, observe that Lemmas \ref{2l1} and \ref{2l2} also imply that,
for any $u\in\Pz$,
\begin{align*}
\sum_{n\in\Lz}\lf|\widehat{\psi^{(u)}}(\oz\pz n)\r|^2=1\ {\rm for\ almost\ every}\ \oz\in G^*.
\end{align*}
By this and a proof similar to \eqref{2e10}, we conclude that \eqref{2e8} is valid; the details
are omitted.

Next, we prove \eqref{2e9}. Indeed, by the fact that $\psi^{(u)}\perp V_0$ with $u\in\Pz$
and an argument similar to \eqref{2e3}, we find that, for any $u\in\Pz$ and $k\in\Lz$,
\begin{align*}
0=\lf(\phi(\cdot\pz k),\psi^{(u)}(\cdot)\r)=\int_{U^*}\sum_{n\in\Lz}
\widehat{\phi}(\oz\pz n)\overline{\widehat{\psi^{(u)}}(\oz\pz n)}\chi(k,\oz)d\mu^*(\oz).
\end{align*}
Therefore,
$${\rm all\ the\ Fourier\ coefficients\ of\ the\ function}\
\sum_{n\in\Lz}\widehat{\phi}(\oz\pz n)\overline{\widehat{\psi^{(u)}}(\oz\pz n)}\
{\rm equal\ to}\ 0.$$
From this, Lemma \ref{2l3} and the $\Lz$-periodicity of $\{m_0,m_1^{(u)}\}_{u\in\Pz}$,
we infer that, for any $u\in\Pz$ and almost every $\oz\in G^*$,
\begin{align*}
0&=\sum_{n\in\Lz}\widehat{\phi}(\rho(\oz)\pz n)\overline{\widehat{\psi^{(u)}}(\rho(\oz)\pz n)}\\
&=\sum_{n\in\Lz}m_0(\sz(\rho(\oz)\pz n))\widehat{\phi}(\sz(\rho(\oz)\pz n))
\overline{m_1^{(u)}(\sz(\rho(\oz)\pz n))}\,\overline{\widehat{\phi}(\sz(\rho(\oz)\pz n))}\\
&=\sum_{n\in\Lz}m_0(\oz\pz\sz(n))\widehat{\phi}(\oz\pz\sz(n))
\overline{m_1^{(u)}(\oz\pz\sz(n))}\,\overline{\widehat{\phi}(\oz\pz\sz(n))}\\
&=\sum_{n\in\Lz_1}m_0(\oz\pz\sz(n))\overline{m_1^{(u)}(\oz\pz\sz(n))}
\lf|\widehat{\phi}(\oz\pz\sz(n))\r|^2\\
&\quad+\sum_{x\in\Pz}\sum_{n\in\Lz_1}m_0(\oz\pz\sz(n)\pz0.x)\overline{m_1^{(u)}(\oz\pz\sz(n)\pz0.x)}
\lf|\widehat{\phi}(\oz\pz\sz(n)\pz0.x)\r|^2\\
&=m_0(\oz)\overline{m_1^{(u)}(\oz)}\sum_{n\in\Lz}\lf|\widehat{\phi}(\oz\pz n)\r|^2
+\sum_{x\in\Pz}m_0(\oz\pz0.x)\overline{m_1^{(u)}(\oz\pz0.x)}
\sum_{n\in\Lz}\lf|\widehat{\phi}(\oz\pz n\pz0.x)\r|^2\\
&=m_0(\oz)\overline{m_1^{(u)}(\oz)}+\sum_{x\in\Pz}m_0(\oz\pz0.x)\overline{m_1^{(u)}(\oz\pz0.x)}.
\end{align*}
Thus, \eqref{2e9} is true.

Finally, by the mutual orthogonality of $\{\psi^{(u)}\}_{u\in\Pz}$, the validity of \eqref{2e9'}
can be easily checked via repeating the proof of \eqref{2e9} with some slight modifications;
the details are omitted. This finishes the proof of Lemma \ref{2l4}.
\end{proof}

\begin{remark}\label{2r}
When $p=2$, the Vilenkin group $G$ used in present article goes back to the Cantor dyadic group
(see, for instance, \cite{Lang96,ms21}). In this case,
Lemmas \ref{2l2} through \ref{2l4} above are, respectively, \cite[Theorems 1 through 3]{ms21}.
\end{remark}

\section{A characterization of wavelet sets on Vilenkin groups}\label{s3}

In this section, we give a necessary and sufficient condition for a set of the dual group of
a Vilenkin group to become a wavelet set.

We begin with the following notion of multiwavelet sets from \cite[Definition 2.2.]{mss23}.

\begin{definition}\label{3d1}
Let $G$ be a Vilenkin group. A set of measurable subsets of $G^*$, denoted by
$$\Oz:=\lf\{\Oz^{(u)}\subset G^*:\ u\in\Pz\r\},$$
is called a \emph{multiwavelet set} if the set
$$\lf\{\psi_{j,n}^{(u)}\r\}_{j\in\zz,n\in\Lz,u\in\Pz}
:=\lf\{\lf(\lf(\mathbf{1}_{\Oz^{(u)}}\r)^\vee\r)_{j,n}\r\}_{j\in\zz,n\in\Lz,u\in\Pz}$$
forms an orthonormal basis of $L^2(G)$.
\end{definition}

Throughout this article, we also call a multiwavelet set simply by a wavelet set in the absence
of confusions.

The succeeding theorem is the first main result of this article, which also can be found in
\cite[Theorem 2.4]{mss23}. However, the authors of \cite{mss23} did not give its proof. For the
sake of completeness and to support our claim in Theorem \ref{t2} below, detailed proofs are given
here via borrowing some ideas from the proof of \cite[Theorem 4]{ms21}.

\begin{theorem}\label{t1}
Let $G$ be a Vilenkin group and  $\Oz=\{\Oz^{(u)}\}_{u\in\Pz}$ a set of measurable subsets
of $G^*$. Then $\Oz$ is a wavelet set if and only if both of the following conditions hold true:
\begin{enumerate}
\item [{\rm (i)}]
$\{\sz^k(\Oz^{(u)})\}_{k\in\zz,u\in\Pz}$ tiles $G^*$ up to sets of measure zero, that is,
for any $u_1,u_2\in\Pz$ and $j,k\in\zz$ with $u_1\neq u_2$ or $j\neq k$,
$$\mu^*\lf(\sz^j\lf(\Oz^{(u_1)}\r)\cap\sz^k\lf(\Oz^{(u_2)}\r)\r)=0$$
and
$$\mu^*\lf(G^*\setminus\bigcup_{k\in\zz,u\in\Pz}\sz^k\lf(\Oz^{(u)}\r)\r)=0.$$
\item [{\rm (ii)}]
For any $u\in\Pz$, $\Oz^{(u)}$ is $\Lz$-translation congruent to $U^*$ up to sets of measure zero,
that is, there exists a partition $\{\Oz_n^{(u)}:\ n\in N^{(u)}\subset\zz_+\}$ of $\Oz^{(u)}$ such that,
for any $n_1,n_2\in N^{(u)}$ with $n_1\neq n_2$,
\begin{align}\label{3e4}
\mu^*\lf(\lf(\Oz_{n_1}^{(u)}\mz n_1\r)\cap\lf(\Oz_{n_2}^{(u)}\mz n_2\r)\r)=0,
\end{align}
\begin{align}\label{3e5}
\bigcup_{n\in N^{(u)}}\lf(\Oz_n^{(u)}\mz n\r)\subset U^*\quad  and\quad
\mu^*\lf(U^*\setminus\bigcup_{n\in N^{(u)}}\lf(\Oz_n^{(u)}\mz n\r)\r)=0.
\end{align}
\end{enumerate}
\end{theorem}

\begin{proof}
We prove the present theorem by two steps.

\emph{Step 1}. In this step, we show the sufficiency. For this purpose, let
$\{\Oz^{(u)}\}_{u\in\Pz}$ be a set of measurable subsets of $G^*$ satisfying both conditions (i) and (ii).
Then, by (ii), we find that, for any $u\in\Pz$, there exists a partition $\{\Oz_n^{(u)}:\ n\in N^{(u)}\subset\zz_+\}$ of $\Oz^{(u)}$ such that both \eqref{3e4} and \eqref{3e5} hold true.
Therefore,
\begin{align}\label{3e1}
\mu^*\lf(\Oz^{(u)}\r)=\sum_{n\in N^{(u)}}\mu^*\lf(\Oz_n^{(u)}\r)
=\sum_{n\in N^{(u)}}\mu^*\lf(\Oz_n^{(u)}\mz n\r)=\mu^*\lf(U^*\r)=1,
\end{align}
which further implies that $\mathbf{1}_{\Oz^{(u)}}\in L^2(G^*)$.
From this and the Parseval equality (see, for instance, \cite[Proposition 1(c)]{Far07}),
it follows that $\|\psi^{(u)}\|=1$ for any $u\in\Pz$.

Next, we prove that the set $\{\psi_{j,n}^{(u)}\}_{j\in\zz,n\in\Lz,u\in\Pz}$ given as in
Definition \ref{3d1} forms an orthonormal basis of $L^2(G)$ by two substeps.

\emph{Substep 1.1}. In this substep, we give the proof of the mutual orthogonality of
$\{\psi_{j,n}^{(u)}\}_{j\in\zz,n\in\Lz,u\in\Pz}$. To this end, by the definition of Fourier
transforms, we obtain that, for any $u\in\Pz$, $j\in\zz$, $n\in\Lz$ and $\oz\in G^*$,
\begin{align}\label{3e2'}
\lf(\psi_{j,n}^{(u)}\r)^\wedge(\oz)
&=\int_{G}\psi_{j,n}^{(u)}(x)\overline{\chi(x,\oz)}\,d\mu(x)\\
&=\int_{G}p^{j/2}\psi^{(u)}\lf(\rho^j(x)\mz n\r)\overline{\chi(x,\oz)}\,d\mu(x)\noz\\
&=p^{-j/2}\int_{G}\psi^{(u)}(y)\overline{\chi\lf(\sz^j(y\pz n),\oz\r)}\,d\mu(y)\noz\\
&=p^{-j/2}\overline{\chi\lf(\sz^j(n),\oz\r)}\lf(\psi^{(u)}\r)^\wedge\lf(\sz^j(\oz)\r).\noz
\end{align}
By this, the Parseval equality again and \eqref{3e1}, we conclude that, for any
$u\in\Pz$, $j\in\zz$ and $n\in\Lz$,
\begin{align}\label{3e2}
\lf\|\psi_{j,n}^{(u)}\r\|^2&=\lf\|\lf(\psi_{j,n}^{(u)}\r)^\wedge\r\|^2
=\int_{G^*}\lf|\lf(\psi_{j,n}^{(u)}\r)^\wedge(\oz)\r|^2\,d\mu^*(\oz)\\
&=p^{-j}\int_{G^*}\lf|\chi\lf(\sz^j(n),\oz\r)\r|^2
\lf|\lf(\psi^{(u)}\r)^\wedge\lf(\sz^j(\oz)\r)\r|^2\,d\mu^*(\oz)\noz\\
&=\int_{G^*}\mathbf{1}_{\Oz^{(u)}}(\oz)\,d\mu^*(\oz)=\mu^*\lf(\Oz^{(u)}\r)=1.\noz
\end{align}

On the other hand, for any $u_1,u_2\in\Pz$, $j,k\in\zz$ and $r,m\in\Lz$, using \eqref{3e2'},
we have
\begin{align}\label{3e3}
&\lf(\psi_{j,r}^{(u_1)},\psi_{k,m}^{(u_2)}\r)\\
&\quad=\lf(\lf(\psi_{j,r}^{(u_1)}\r)^\wedge,\lf(\psi_{k,m}^{(u_2)}\r)^\wedge\r)\noz\\
&\quad=\int_{G^*}\lf(\psi_{j,r}^{(u_1)}\r)^\wedge(\oz)
\overline{\lf(\psi_{k,m}^{(u_2)}\r)^\wedge(\oz)}\,d\mu^*(\oz)\noz\\
&\quad=\int_{G^*}p^{-(j+k)/2}\overline{\chi\lf(\sz^j(r),\oz\r)}\lf(\psi^{(u_1)}\r)^\wedge
\lf(\sz^j(\oz)\r)\chi\lf(\sz^k(m),\oz\r)\overline{\lf(\psi^{(u_2)}\r)^\wedge\lf(\sz^k(\oz)\r)}
\,d\mu^*(\oz)\noz\\
&\quad=\int_{\rho^j\Oz^{(u_1)}\cap\rho^k\Oz^{(u_2)}}p^{-(j+k)/2}
\overline{\chi\lf(\sz^j(r),\oz\r)}\chi\lf(\sz^k(m),\oz\r)\,d\mu^*(\oz).\noz
\end{align}
From this and the condition (i), we deduce that,
for any $u_1,u_2\in\Pz$ and $j,k\in\zz$ with $u_1\neq u_2$ or $j\neq k$,
\begin{align}\label{3e3'}
\lf(\psi_{j,r}^{(u_1)},\psi_{k,m}^{(u_2)}\r)=0.
\end{align}
If $u_1=u_2=u\in\Pz$ and $j=k\in\zz$, then \eqref{3e3} takes the form
\begin{align}\label{3e6}
\lf(\psi_{j,r}^{(u)},\psi_{j,m}^{(u)}\r)
&=\int_{\rho^j\Oz^{(u)}}p^{-j}\overline{\chi\lf(\sz^j(r),\oz\r)}\chi\lf(\sz^j(m),\oz\r)\,d\mu^*(\oz)\\
&=\int_{\Oz^{(u)}}\overline{\chi\lf(r,\oz\r)}\chi\lf(m,\oz\r)\,d\mu^*(\oz).\noz
\end{align}
Observe that, for any $n\in\Lz$ and $\oz\in G^*$, $\oz\mz pn=\oz$. Then, for any
$u\in\Pz$ and $n\in N^{(u)}$,
\begin{align*}
\int_{\Oz_n^{(u)}}\overline{\chi\lf(r,\oz\r)}\chi\lf(m,\oz\r)\,d\mu^*(\oz)
&=\frac1{p^2}\int_{\Oz_n^{(u)}}\overline{\chi\lf(r,p\oz\r)}\chi\lf(m,p\oz\r)\,d\mu^*(\oz)\\
&=\frac1{p^2}\int_{\Oz_n^{(u)}}\overline{\chi\lf(r,p\oz\mz pn\r)}\chi\lf(m,p\oz\mz pn\r)\,d\mu^*(\oz)\\
&=\int_{\Oz_n^{(u)}}\overline{\chi\lf(r,\oz\mz n\r)}\chi\lf(m,\oz\mz n\r)\,d\mu^*(\oz)\\
&=\int_{\Oz_n^{(u)}\mz n}\overline{\chi\lf(r,\oz\r)}\chi\lf(m,\oz\r)\,d\mu^*(\oz).
\end{align*}
This, together with \eqref{3e6}, the fact that $\{\Oz_n^{(u)}:\ n\in N^{(u)}\}$ is a
partition of $\Oz^{(u)}$, \eqref{3e4}, \eqref{3e5} and the mutual orthogonality of
$\{\chi(r,\cdot)\}_{r\in\Lz}$ on $U^*$ (see, for instance, \cite[p.\,311]{Far08}), implies that
\begin{align*}
\lf(\psi_{j,r}^{(u)},\psi_{j,m}^{(u)}\r)
&=\sum_{n\in N^{(u)}}\int_{\Oz_n^{(u)}}\overline{\chi\lf(r,\oz\r)}\chi\lf(m,\oz\r)\,d\mu^*(\oz)\\
&=\sum_{n\in N^{(u)}}\int_{\Oz_n^{(u)}\mz n}\overline{\chi\lf(r,\oz\r)}\chi\lf(m,\oz\r)\,d\mu^*(\oz)\\
&=\int_{U^*}\overline{\chi\lf(r,\oz\r)}\chi\lf(m,\oz\r)\,d\mu^*(\oz)\\
&=\begin{cases}
1
&\text{if}\ r=m,\\
0
&\text{if}\ r\neq m.
\end{cases}
\end{align*}
By this, \eqref{3e2}
and \eqref{3e3'}, it is easy to see that $\{\psi_{j,n}^{(u)}\}_{j\in\zz,n\in\Lz,u\in\Pz}$ is an
orthonormal set and hence the proof of Substep 1.1 is completed.

\emph{Substep 1.2}. In this substep, we show that the system
$\{\psi_{j,n}^{(u)}\}_{j\in\zz,n\in\Lz,u\in\Pz}$ is complete in $L^2(G)$, that is, for any $f\in L^2(G)$,
\begin{align*}
\sum_{j\in\zz,n\in\Lz,u\in\Pz}\lf|\lf(f,\psi_{j,n}^{(u)}\r)\r|^2=\|f\|^2.
\end{align*}
To do this, using the Parseval equality (see, for instance, \cite[Proposition 1(c)]{Far07})
and \eqref{3e2'}, we know that, for any $f\in L^2(G)$,
\begin{align*}
\sum_{j\in\zz,n\in\Lz,u\in\Pz}\lf|\lf(f,\psi_{j,n}^{(u)}\r)\r|^2
&=\sum_{j\in\zz,n\in\Lz,u\in\Pz}\lf|\lf(\widehat{f},\lf(\psi_{j,n}^{(u)}\r)^\wedge\r)\r|^2\\
&=\sum_{j\in\zz,n\in\Lz,u\in\Pz}\lf|\int_{G^*}\widehat{f}(\oz)
p^{-j/2}\chi\lf(\sz^j(n),\oz\r)\overline{\lf(\psi^{(u)}\r)^\wedge\lf(\sz^j(\oz)\r)}\,d\mu^*(\oz)\r|^2\\
&=\sum_{j\in\zz,n\in\Lz,u\in\Pz}\lf|\int_{G^*}\widehat{f}\lf(\rho^j(\oz)\r)
p^{j/2}\chi\lf(n,\oz\r)\overline{\lf(\psi^{(u)}\r)^\wedge\lf(\oz\r)}\,d\mu^*(\oz)\r|^2\\
&=\sum_{j\in\zz,n\in\Lz,u\in\Pz}\lf|\int_{\Oz^{(u)}}\widehat{f}\lf(\rho^j(\oz)\r)
p^{j/2}\chi\lf(n,\oz\r)\,d\mu^*(\oz)\r|^2.
\end{align*}
From this, the fact that $\{\Oz_n^{(u)}:\ n\in N^{(u)}\}$ is a
partition of $\Oz^{(u)}$, \eqref{3e4} and \eqref{3e5}, we infer that, for any $f\in L^2(G)$,
\begin{align}\label{3e7}
&\sum_{j\in\zz,n\in\Lz,u\in\Pz}\lf|\lf(f,\psi_{j,n}^{(u)}\r)\r|^2\\
&\quad=\sum_{j\in\zz,n\in\Lz,u\in\Pz}\lf|\sum_{i\in N^{(u)}}\int_{\Oz_i^{(u)}}
\widehat{f}\lf(\rho^j(\oz)\r)p^{j/2}\chi\lf(n,\oz\r)\,d\mu^*(\oz)\r|^2\noz\\
&\quad=\sum_{j\in\zz,n\in\Lz,u\in\Pz}p^{j}\lf|\sum_{i\in N^{(u)}}\int_{U^*}
\mathbf{1}_{\Oz_i^{(u)}\mz i}(\oz)\widehat{f}\lf(\rho^j(\oz\pz i\r)
\chi\lf(n,\oz\pz i\r)\,d\mu^*(\oz)\r|^2\noz\\
&\quad=\sum_{j\in\zz,n\in\Lz,u\in\Pz}p^{j}\lf|\sum_{i\in N^{(u)}}\frac1p\int_{U^*}
\mathbf{1}_{\Oz_i^{(u)}\mz i}(\oz)\widehat{f}\lf(\rho^j(\oz\pz i\r)
\chi\lf(n,p\oz\pz pi\r)\,d\mu^*(\oz)\r|^2\noz\\
&\quad=\sum_{j\in\zz,n\in\Lz,u\in\Pz}p^{j}\lf|\sum_{i\in N^{(u)}}\int_{U^*}
\mathbf{1}_{\Oz_i^{(u)}\mz i}(\oz)\widehat{f}\lf(\rho^j(\oz\pz i\r)
\chi\lf(n,\oz\r)\,d\mu^*(\oz)\r|^2\noz\\
&\quad=\sum_{j\in\zz,u\in\Pz}p^{j}\sum_{n\in\Lz}\lf|\int_{U^*}\sum_{i\in N^{(u)}}
\mathbf{1}_{\Oz_i^{(u)}\mz i}(\oz)\widehat{f}\lf(\rho^j(\oz\pz i\r)
\chi\lf(n,\oz\r)\,d\mu^*(\oz)\r|^2.\noz
\end{align}

Note that, for any $u\in\Pz$ and $j\in\zz$,
$$\supp\lf[\sum_{i\in N^{(u)}}
\mathbf{1}_{\Oz_i^{(u)}\mz i}(\oz)\widehat{f}\lf(\rho^j(\oz\pz i\r)\r]\subset U^*.$$
Then, for any $\oz\notin U^*$, we have $\oz\notin (\Oz_i^{(u)}\mz i)$ and hence
$\mathbf{1}_{\Oz_i^{(u)}\mz i}(\oz)=0$ for every $i\in N^{(u)}$.
Therefore, for any $u\in\Pz$, $j\in\zz$ and $\oz\notin U^*$,
$$\sum_{i\in N^{(u)}}\mathbf{1}_{\Oz_i^{(u)}\mz i}(\oz)\widehat{f}\lf(\rho^j(\oz\pz i\r)=0,$$
which, combined with the Plancherel theorem for Fourier series, further implies that
\begin{align*}
&\sum_{n\in\Lz}\lf|\int_{U^*}\sum_{i\in N^{(u)}}
\mathbf{1}_{\Oz_i^{(u)}\mz i}(\oz)\widehat{f}\lf(\rho^j(\oz\pz i\r)
\chi\lf(n,\oz\r)\,d\mu^*(\oz)\r|^2\\
&\quad=\int_{G^*}\lf|\sum_{i\in N^{(u)}}
\mathbf{1}_{\Oz_i^{(u)}\mz i}(\oz)\widehat{f}\lf(\rho^j(\oz\pz i\r)\r|^2\,d\mu^*(\oz)\\
&\quad=\int_{U^*}\lf|\sum_{i\in N^{(u)}}
\mathbf{1}_{\Oz_i^{(u)}\mz i}(\oz)\widehat{f}\lf(\rho^j(\oz\pz i\r)\r|^2\,d\mu^*(\oz).
\end{align*}
From this and \eqref{3e7} as well as conditions (i) and (ii), it follows that, for any $f\in L^2(G)$,
\begin{align*}
\sum_{j\in\zz,n\in\Lz,u\in\Pz}\lf|\lf(f,\psi_{j,n}^{(u)}\r)\r|^2
&=\sum_{j\in\zz,u\in\Pz}p^{j}\int_{U^*}\lf|\sum_{i\in N^{(u)}}
\mathbf{1}_{\Oz_i^{(u)}\mz i}(\oz)\widehat{f}\lf(\rho^j(\oz\pz i\r)\r|^2\,d\mu^*(\oz)\\
&=\sum_{j\in\zz,u\in\Pz}p^{j}\int_{U^*}\sum_{i\in N^{(u)}}
\lf|\widehat{f}\lf(\rho^j(\oz\pz i\r)\r|^2\mathbf{1}_{\Oz_i^{(u)}\mz i}(\oz)\,d\mu^*(\oz)\\
&=\sum_{j\in\zz,u\in\Pz}p^{j}\sum_{i\in N^{(u)}}\int_{\Oz_i^{(u)}}
\lf|\widehat{f}\lf(\rho^j(\oz\r)\r|^2\,d\mu^*(\oz)\\
&=\sum_{j\in\zz,u\in\Pz}\int_{\rho^j(\Oz^{(u)})}\lf|\widehat{f}(\oz)\r|^2\,d\mu^*(\oz)\\
&=\int_{G^*}\lf|\widehat{f}(\oz)\r|^2\,d\mu^*(\oz)
=\lf\|\widehat{f}\,\r\|^2=\|f\|^2.
\end{align*}
This finishes the proof of Substep 1.2 and hence of Step 1, that is, the sufficiency of Theorem \ref{t1}
has been proved.

\emph{Step 2}. In this step, we show the necessity. For this purpose,
assume that $\Oz=\{\Oz^{(u)}\}_{u\in\Pz}$ is a wavelet set, namely, the set
\begin{align}\label{3e8}
\lf\{\psi_{j,n}^{(u)}\r\}_{j\in\zz,n\in\Lz,u\in\Pz}
:=\lf\{\lf(\lf(\mathbf{1}_{\Oz^{(u)}}\r)^\vee\r)_{j,n}\r\}_{j\in\zz,n\in\Lz,u\in\Pz}
\end{align}
forms an orthonormal basis of $L^2(G)$. Then, using the Parseval equality
(see, for instance, \cite[Proposition 1(c)]{Far07}), we find that, for any $u\in\Pz$,
\begin{align}\label{3e8'}
\mu^*\lf(\Oz^{(u)}\r)=\lf\|\lf(\mathbf{1}_{\Oz^{(u)}}\r)^\vee\r\|
=\lf\|\mathbf{1}_{\Oz^{(u)}}\r\|=\lf\|\psi^{(u)}\r\|=1.
\end{align}

To verify (i), by the definition \eqref{3e8}, we obtain that, for any $u_1,u_2\in\Pz$ and $j,k\in\zz$,
\begin{align*}
\mu^*\lf(\rho^j\lf(\Oz^{(u_1)}\r)\cap\rho^k\lf(\Oz^{(u_2)}\r)\r)
=\int_{G^*}\mathbf{1}_{\rho^j(\Oz^{(u_1)})}(\oz)
\overline{\mathbf{1}_{\rho^k(\Oz^{(u_2)})}(\oz)}\,d\mu^*(\oz)\\
=\int_{G^*}\mathbf{1}_{\Oz^{(u_1)}}\lf(\sz^j(\oz)\r)
\overline{\mathbf{1}_{\Oz^{(u_2)}}\lf(\sz^k(\oz)\r)}\,d\mu^*(\oz)\\
=\int_{G^*}\widehat{\psi^{(u_1)}}\lf(\sz^j(\oz)\r)
\overline{\widehat{\psi^{(u_2)}}\lf(\sz^k(\oz)\r)}\,d\mu^*(\oz).
\end{align*}
Observe that, for any $u\in\Pz$, $j\in\zz$ and $\oz\in G^*$,
\begin{align}\label{3e9'}
\widehat{\psi^{(u)}}\lf(\sz^j(\oz)\r)
&=\int_{G}\psi^{(u)}(x)\overline{\chi\lf(x,\sz^j(\oz)\r)}\,d\mu(x)\\
&=\int_{G}\psi^{(u)}(x)\overline{\chi\lf(\sz^j(x),\oz\r)}\,d\mu(x)\noz\\
&=p^{j}\int_{G}\psi^{(u)}\lf(\rho^j(x)\r)\overline{\chi(x,\oz)}\,d\mu(x)\noz\\
&=p^{j/2}\widehat{\psi^{(u)}_{j,0}}(\oz).\noz
\end{align}
Then, by this, the Parseval equality and the mutual orthogonality of
$\{\psi_{j,0}^{(u)}\}_{j\in\zz,u\in\Pz}$, we conclude that, for any $u_1,u_2\in\Pz$ and $j,k\in\zz$
with $u_1\neq u_2$ or $j\neq k$,
\begin{align*}
\mu^*\lf(\rho^j\lf(\Oz^{(u_1)}\r)\cap\rho^k\lf(\Oz^{(u_2)}\r)\r)
&=p^{(j+k)/2}\int_{G^*}\widehat{\psi^{(u_1)}_{j,0}}(\oz)
\overline{\widehat{\psi^{(u_2)}_{j,0}}(\oz)}\,d\mu^*(\oz)\\
&=p^{(j+k)/2}\int_{G}\psi^{(u_1)}_{j,0}(x)\overline{\psi^{(u_2)}_{j,0}(x)}\,d\mu(x)
=0.
\end{align*}

On another hand, for any
$$\az\in G^*\setminus\bigcup_{k\in\zz,u\in\Pz}\sz^k\lf(\Oz^{(u)}\r)=:E,$$
let $E_\az:=E\cap(\az\pz U^*)$. Then $\mathbf{1}_{E_\az}\in L^2(G^*)$. From this,
the Parseval equality and the fact that
$\{\psi_{j,n}^{(u)}\}_{j\in\zz,n\in\Lz,u\in\Pz}$ is an orthonormal basis of $L^2(G)$,
we deduce that, for any $\az\in E$,
\begin{align}\label{3e9}
\mu^*\lf(E_\az\r)&=\lf\|\mathbf{1}_{E_\az}\r\|=\lf\|\lf(\mathbf{1}_{E_\az}\r)^\vee\r\|\\
&=\sum_{j\in\zz,n\in\Lz,u\in\Pz}\lf|\lf(\lf(\mathbf{1}_{E_\az}\r)^\vee,\psi_{j,n}^{(u)}\r)\r|^2\noz\\
&=\sum_{j\in\zz,n\in\Lz,u\in\Pz}\lf|\lf(\mathbf{1}_{E_\az},\widehat{\psi_{j,n}^{(u)}}\r)\r|^2\noz\\
&=\sum_{j\in\zz,n\in\Lz,u\in\Pz}\lf|\int_{G^*}\mathbf{1}_{E_\az}(\oz)
\overline{\widehat{\psi_{j,n}^{(u)}}(\oz)}\,d\mu^*(\oz)\r|^2.\noz
\end{align}
Note that, for any $u\in\Pz$, $j\in\zz$, $n\in\Lz$ and $\oz\in G^*$,
\begin{align}\label{3e10'}
\widehat{\psi^{(u)}_{j,n}}(\oz)
&=p^{j/2}\int_{G}\psi^{(u)}\lf(\rho^j(x)\mz n\r)\overline{\chi(x,\oz)}\,d\mu(x)\\
&=p^{-j/2}\int_{G}\psi^{(u)}(x)\overline{\chi\lf(\sz^j(x\pz n\r),\oz)}\,d\mu(x)\noz\\
&=p^{-j/2}\int_{G}\psi^{(u)}(x)
\overline{\chi\lf(x,\sz^j(\oz)\r)}\,\overline{\chi\lf(n,\sz^j(\oz)\r)}\,d\mu(x)\noz\\
&=p^{-j/2}\overline{\chi\lf(n,\sz^j(\oz)\r)}\widehat{\psi^{(u)}}\lf(\sz^j\oz\r).\noz
\end{align}
Then \eqref{3e9} takes the form
\begin{align*}
\mu^*\lf(E_\az\r)&=\sum_{j\in\zz,n\in\Lz,u\in\Pz}\lf|\int_{G^*}\mathbf{1}_{E_\az}(\oz)
p^{-j/2}\chi\lf(n,\sz^j(\oz)\r)\overline{\widehat{\psi^{(u)}}\lf(\sz^j(\oz)\r)}\,d\mu^*(\oz)\r|^2\\
&=\sum_{j\in\zz,n\in\Lz,u\in\Pz}p^{-j}\lf|\int_{G^*}\mathbf{1}_{E_\az}(\oz)
\chi\lf(n,\sz^j(\oz)\r)\mathbf{1}_{\rho^j(\Oz^{(u)})}(\oz)\,d\mu^*(\oz)\r|^2\\
&=\sum_{j\in\zz,n\in\Lz,u\in\Pz}p^{-j}\lf|\int_{E_\az\cap\rho^j(\Oz^{(u)})}
\chi\lf(n,\sz^j(\oz)\r)\,d\mu^*(\oz)\r|^2=0,
\end{align*}
where we have used the fact that $E_\az\cap\rho^j(\Oz^{(u)})=\emptyset$ for any $u\in\Pz$ and $j\in\zz$
to obtain the last equality. This, together with the $\sz$-compactness of $G^*$, implies that
\begin{align*}
\mu^*(E)\le\mu^*\lf(\bigcup_{\az\in E}E_\az\r)\le\sum_{\az\in E}\mu^*\lf(E_\az\r)=0.
\end{align*}
Thus, $\mu^*(E)=0$, which completes the proof of (i).

Finally, we prove (ii). To this end, for any $u\in\Pz$, let $\{n:\ n\in N^{(u)}\subset\zz_+\}$
be the set of all non-negative integers such that, for the corresponding $n\in\Lz$,
the set $\Oz_n^{(u)}:=\Oz^{(u)}\cap(U^*\pz n)$ has positive measure. Then
$\{\Oz_n^{(u)}:\ n\in N^{(u)}\}$ is a partition of $\Oz^{(u)}$ for any $u\in\Pz$. To show
\eqref{3e4} and \eqref{3e5}, it suffices to prove that, for any $u\in\Pz$ and almost every $\oz\in U^*$,
\begin{align}\label{3e11}
F^{(u)}(\oz):=\sum_{n\in N^{(u)}}\mathbf{1}_{\Oz_n^{(u)}\mz n}(\oz)=1
\end{align}
holds true. To verify \eqref{3e11}, by the Parseval equality and \eqref{3e10'}, we find that,
for any $u\in\Pz$, $j\in\zz$ and $m\in\Lz$,
\begin{align}\label{3e12}
\lf(\psi^{(u)}_{j,\theta},\psi^{(u)}_{j,m}\r)
&=\lf(\widehat{\psi^{(u)}_{j,\theta}},\widehat{\psi^{(u)}_{j,m}}\r)\\
&=\int_{G^*}\widehat{\psi^{(u)}_{j,\theta}}(\oz)\overline{\widehat{\psi^{(u)}_{j,m}}(\oz)}\,d\mu^*(\oz)\noz\\
&=\int_{G^*}p^{-j}\overline{\chi\lf(\theta,\sz^j(\oz)\r)}\widehat{\psi^{(u)}}\lf(\sz^j\oz\r)
\chi\lf(m,\sz^j(\oz)\r)\overline{\widehat{\psi^{(u)}}\lf(\sz^j\oz\r)}\,d\mu^*(\oz)\noz\\
&=\int_{G^*}p^{-j}\lf|\widehat{\psi^{(u)}}\lf(\sz^j\oz\r)\r|^2\chi\lf(m,\sz^j(\oz)\r)\,d\mu^*(\oz)\noz\\
&=\int_{G^*}\lf|\widehat{\psi^{(u)}}(\oz)\r|^2\chi\lf(m,\oz\r)\,d\mu^*(\oz)\noz\\
&=\int_{\Oz^{(u)}}\chi\lf(m,\oz\r)\,d\mu^*(\oz)\noz\\
&=\sum_{n\in N^{(u)}}\int_{\Oz_n^{(u)}}\chi\lf(m,\oz\r)\,d\mu^*(\oz)\noz\\
&=\sum_{n\in N^{(u)}}\frac1p\int_{\Oz_n^{(u)}}\chi\lf(m,p\oz\mz pn\r)\,d\mu^*(\oz)\noz\\
&=\sum_{n\in N^{(u)}}\int_{\Oz_n^{(u)}\mz n}\chi\lf(m,\oz\r)\,d\mu^*(\oz),\noz
\end{align}
where $\theta:=(\ldots,0,0,\ldots)$ denotes the identity element of $G^*$.

In addition, for any $u\in\Pz$ and almost every $\oz\in U^*$, let
\begin{align*}
F_M^{(u)}(\oz):=\sum_{n\in N_M^{(u)}}\mathbf{1}_{\Oz_n^{(u)}\mz n}(\oz),
\end{align*}
where $\{N_M^{(u)}\}$ is a family of finite subsets of $N^{(u)}$ such that
$N_M^{(u)}\to N^{(u)}$ as $M\to+\fz$. Then we have
\begin{align*}
\int_{U^*}F_M^{(u)}(\oz)\,d\mu^*(\oz)
&=\sum_{n\in N_M^{(u)}}\int_{U^*}\mathbf{1}_{\Oz_n^{(u)}\mz n}(\oz)\,d\mu^*(\oz)\\
&=\sum_{n\in N_M^{(u)}}\mu^*\lf(\Oz_n^{(u)}\mz n\r)=\sum_{n\in N_M^{(u)}}\mu^*\lf(\Oz_n^{(u)}\r).
\end{align*}
From this, the fact that both
$$0\le F_M^{(u)}(\oz)\le F^{(u)}(\oz)\quad {\rm and}\quad\lim_{M\to+\fz}F_M^{(u)}(\oz)=F^{(u)}(\oz)$$
hold true for any $u\in\Pz$ and almost every $\oz\in U^*$ as well as \eqref{3e8'}, we infer that
\begin{align*}
\int_{U^*}F^{(u)}(\oz)\,d\mu^*(\oz)
&=\lim_{M\to+\fz}\int_{U^*}F_M^{(u)}(\oz)\,d\mu^*(\oz)\\
&=\lim_{M\to+\fz}\sum_{n\in N_M^{(u)}}\mu^*\lf(\Oz_n^{(u)}\r)\\
&=\sum_{n\in N^{(u)}}\mu^*\lf(\Oz_n^{(u)}\r)=\mu^*\lf(\Oz^{(u)}\r)=1.
\end{align*}
Therefore, for any $u\in\Pz$, $F^{(u)}\in L^1(U^*)$. By this, \eqref{3e12} and the Lebesgue dominated
convergence theorem, we conclude that, for any $u\in\Pz$, $j\in\zz$ and $m\in\Lz$,
\begin{align*}
\lf(\psi^{(u)}_{j,\theta},\psi^{(u)}_{j,m}\r)
&=\lim_{M\to+\fz}\sum_{n\in N_M^{(u)}}\int_{\Oz_n^{(u)}\mz n}(\oz)\chi\lf(m,\oz\r)\,d\mu^*(\oz)\\
&=\lim_{M\to+\fz}\sum_{n\in N_M^{(u)}}
\int_{U^*}\mathbf{1}_{\Oz_n^{(u)}\mz n}\chi\lf(m,\oz\r)\,d\mu^*(\oz)\\
&=\lim_{M\to+\fz}\int_{U^*}F_M^{(u)}(\oz)\chi\lf(m,\oz\r)\,d\mu^*(\oz)\\
&=\int_{U^*}F^{(u)}(\oz)\chi\lf(m,\oz\r)\,d\mu^*(\oz).
\end{align*}
This, combined with the orthogonality that, for any $u\in\Pz$,
\begin{align*}
\lf(\psi^{(u)}_{j,\theta},\psi^{(u)}_{j,m}\r)
=\begin{cases}
1
&\text{if}\ m=\theta\\
0
&\text{if}\ m\neq \theta
\end{cases}
\end{align*}
and the uniqueness theorem for Fourier series, further implies the validity of \eqref{3e11}.
Thus, we have completed the proof of (ii) and hence of Theorem \ref{t1}.
\end{proof}

\begin{remark}\label{3r}
\begin{enumerate}
\item[{\rm (i)}]
As in Remark \ref{2r}, when $p=2$, the Vilenkin group $G$ is just the Cantor dyadic group.
Therefore, Theorem \ref{t1} includes \cite[Theorem 4]{ms21} as a special case.
\item[{\rm (ii)}]
A nice example of wavelet sets was given in \cite[Example 2.2]{mss23}; More examples of wavelet sets
can be also found in \cite{bb11,ms21}.
\end{enumerate}
\end{remark}

\section{An application to construction of MRA wavelets}\label{s4}

In this section, as an application of Theorem \ref{t1}, we establish the relation between
MRA and wavelets determined from wavelet sets. This provides another method, which is different
from that used in \cite{Far08}, for the construction of wavelets.

To prove the main result of this section, we need the following two technical lemmas.
The first one (see Lemma \ref{4l1} below) gives an equal relation between the associated scaling
function $\phi$ and the sequence $\{\psi^{(u)}\}_{u\in\Pz}$ of wavelets, which is of independent
interest. The second one (see Lemma \ref{4l2} below) shows that the elements of $V_0$ and $V_{-1}$
can be characterized by using the associated scaling filter $m_0$.

\begin{lemma}\label{4l1}
Let $\{V_j\}_{j\in\zz}$ be an {\rm MRA} as in Definition \ref{2d1}. Assume that $\phi$ is a scaling
function and $\{\psi^{(u)}\}_{u\in\Pz}$ a sequence of wavelets associated with $\{V_j\}_{j\in\zz}$.
Then, for almost every $\oz\in G^*$,
\begin{align}\label{4e1}
\lf|\widehat{\phi}(\oz)\r|^2
=\sum_{u\in\Pz}\sum_{j\in\nn}\lf|\widehat{\psi^{(u)}}\lf(\rho^j(\oz)\r)\r|^2.
\end{align}
\end{lemma}

\begin{proof}
If $p=2$, then \eqref{4e1} can be found in \cite[p.\,14]{ms21}.

If $p\in[3,\fz)\cap\nn$, then, by Lemma \ref{2l4}, we know that, for almost every $\oz\in G^*$,
\begin{align*}
A(\oz):=\lf(\begin{array}{ccccccc}
\vspace{0.75cm}
\overline{m_0(\oz)} & m_1^{(1)}(\oz) & \cdots & m_1^{(p-1)}(\oz)\\
\vspace{0.2cm}
\overline{m_0(\oz\pz0.1)} & m_1^{(1)}(\oz\pz0.1) & \cdots & m_1^{(p-1)}(\oz\pz0.1)\\
\vspace{0.2cm}
\vdots & \vdots& & \vdots\\
\overline{m_0(\oz\pz0.p-1)} \quad& m_1^{(1)}(\oz\pz0.p-1) & \cdots & m_1^{(p-1)}(\oz\pz0.p-1)\\
\end{array}\r)
\end{align*}
is a unitary matrix, where $m_0$ and $\{m_1^{(u)}\}_{u\in\Pz}$ are, respectively,
the scaling filter and wavelet filters as in Lemma \ref{2l3}. On another hand, from
the argument in \cite[pp.\,319-323]{Far08}, we deduce that, for any $\oz\in G^*$,
$$|m_0(\oz)|\le\frac1{\sqrt{p}}\quad{\rm and}\quad\lf|m_1^{(u)}(\oz)\r|\le\frac1{\sqrt{p}},
\quad\forall\,u\in\Pz.$$
This, together with the above unitarity of $A(\oz)$, further implies that,
for almost every $\oz\in G^*$,
\begin{align*}
|m_0(\oz)|=\frac1{\sqrt{p}}\quad{\rm and}\quad\lf|m_1^{(u)}(\oz)\r|=\frac1{\sqrt{p}},
\quad\forall\,u\in\Pz.
\end{align*}
By this and Lemma \ref{2l3}, we find that, for almost every $\oz\in G^*$,
\begin{align*}
\lf|\widehat{\phi}(\rho(\oz))\r|^2+\sum_{u\in\Pz}\lf|\widehat{\psi^{(u)}}(\rho(\oz))\r|^2
&=\lf|m_0(\oz)\widehat{\phi}(\oz)\r|^2+
\sum_{u\in\Pz}\lf|m_1^{(u)}(\oz)\widehat{\phi}(\oz)\r|^2\\
&=\lf|\widehat{\phi}(\oz)\r|^2.
\end{align*}
Iterating this relation, we obtain
\begin{align}\label{4e2}
\lf|\widehat{\phi}(\oz)\r|^2
=\lf|\widehat{\phi}\lf(\rho^N(\oz)\r)\r|^2+\sum_{j\in[1,N]\cap\nn}\sum_{u\in\Pz}
\lf|\widehat{\psi^{(u)}}\lf(\rho^j(\oz)\r)\r|^2,\quad\forall\,N\in\nn.
\end{align}

In addition, from Lemma \ref{2l1}, it follows that $\{\phi(\cdot\pz n)\}_{n\in\Lz}$ is an
orthonormal set in $L^2(G)$. This, combined with Lemma \ref{2l2}, implies that
\begin{align*}
\sum_{n\in\Lz}\lf|\widehat{\phi}(\oz\pz n)\r|^2=1\ {\rm for\ almost\ every}\ \oz\in G^*.
\end{align*}
Thus, $|\widehat{\phi}(\oz)|\le1$ for almost every $\oz\in G^*$. By this and \eqref{4e2},
it is easy to see that
$$\lf\{\sum_{j\in[1,N]\cap\nn}\sum_{u\in\Pz}
\lf|\widehat{\psi^{(u)}}\lf(\rho^j(\oz)\r)\r|^2\r\}_{N\in\nn}$$
is an increasing sequence with an upper bound of 1 and hence its limit exists. Therefore,
the sequence $\{|\widehat{\phi}(\rho^N(\oz))|^2\}_{N\in\nn}$ in \eqref{4e2} is also convergent.
From this, the Fatou lemma and the fact that $\widehat{\phi}\in L^2(G^*)$, we infer that
\begin{align*}
\int_{G^*}\lim_{N\to\fz}\lf|\widehat{\phi}\lf(\rho^N(\oz)\r)\r|^2\,d\mu^*(\oz)
&\le\lim_{N\to\fz}\int_{G^*}\lf|\widehat{\phi}\lf(\rho^N(\oz)\r)\r|^2\,d\mu^*(\oz)\\
&=\lim_{N\to\fz}\frac1{2^N}\int_{G^*}\lf|\widehat{\phi}(\oz)\r|^2\,d\mu^*(\oz)\\
&=0.
\end{align*}
By this and \eqref{4e2}, we conclude that \eqref{4e1} holds true for the case when
$p\in[3,\fz)\cap\nn$. Thus, the proof of Lemma \ref{4l1} is finished.
\end{proof}

\begin{lemma}\label{4l2}
Let $\{V_j\}_{j\in\zz}$ be an {\rm MRA} as in Definition \ref{2d1}. Assume that $\phi$ and $m_0$
are, respectively, the associated scaling function and filter. Then, for any $f\in L^2(G)$,
the following two items hold true:
\begin{enumerate}
\item[{\rm (i)}]
$f\in V_0$ if and only if there exists some $\Lz$-periodic function $g\in L^2(U^*)$ such that,
for any $\oz\in U^*$,
$$\widehat{f}(\oz)=g(\oz)\widehat{\phi}(\oz);$$
\item[{\rm (ii)}]
$f\in V_{-1}$ if and only if there exists some $\Lz$-periodic function $h\in L^2(U^*)$ such that,
for any $\oz\in U^*$,
$$\widehat{f}(\oz)=h(\rho(\oz))m_0(\oz)\widehat{\phi}(\oz).$$
\end{enumerate}
\end{lemma}

\begin{proof}
We prove the present lemma by two steps.

\emph{Step 1}. In this step, we show (i). For the necessity, let $f\in V_0$. Note that
Lemma \ref{2l1} implies that the set $\{\phi(\cdot\mz n)\}_{n\in\Lz}$ forms an orthonormal basis
of $V_0$. Then, for any $x\in G$, we have
\begin{align*}
f(x)=\sum_{n\in\Lz}t_n\phi(x\mz n),
\end{align*}
where, for any $n\in\Lz$,
\begin{align*}
t_n:=\int_G f(x)\overline{\phi(x\mz n)}\,d\mu(x).
\end{align*}
Moreover, for any $\oz\in U^*$,
\begin{align}\label{4e3}
\widehat{f}(\oz)&=\int_Gf(x)\overline{\chi(x,\oz)}\,d\mu(x)\\
&=\int_G\sum_{n\in\Lz}t_n\phi(x\mz n)\overline{\chi(x,\oz)}\,d\mu(x)\noz\\
&=\int_G\sum_{n\in\Lz}t_n\phi(z)\overline{\chi(z\pz n,\oz)}\,d\mu(z)\noz\\
&=\int_G\sum_{n\in\Lz}t_n\phi(z)\overline{\chi(z,\oz)}\,\overline{\chi(n,\oz)}\,d\mu(z)\noz\\
&=\sum_{n\in\Lz}t_n\overline{\chi(n,\oz)}\int_G\phi(z)\overline{\chi(z,\oz)}\,d\mu(z)\noz\\
&=g(\oz)\widehat{\phi}(\oz),\noz
\end{align}
where
\begin{align*}
g(\oz):=\sum_{n\in\Lz}t_n\overline{\chi(n,\oz)}
\end{align*}
is a $\Lz$-periodic function which belongs to $L^2(U^*)$, since the system
$\{\chi(n,\cdot)\}_{n\in\Lz}$ is an orthogonal basis of $L^2(U^*)$ (see, for instance,
\cite[p.\,311]{Far08}). The proof of the necessity of (i) is completed.

The sufficiency of (i) can be verified by an argument similar as above; the details are omitted.
This shows (i).

\emph{Step 2}. In this step, we prove (ii). Observe that the set
$\{\frac1{\sqrt{p}}\phi(\sz(\cdot)\mz n)\}_{n\in\Lz}$ forms an orthonormal basis of $V_{-1}$
(see Lemma \ref{2l1}). For any given $f\in V_{-1}$, similarly to \eqref{4e3}, we obtain that,
for any $\oz\in U^*$,
\begin{align*}
\widehat{f}(\oz)
&=\sum_{n\in\Lz}r_n\overline{\chi(\rho(n),\oz)}\int_G\phi(z)\overline{\chi(\rho(z),\oz)}\,d\mu(z)\\
&=h(\rho(\oz))\widehat{\phi}(\rho(\oz))=h(\rho(\oz))m_0(\oz)\widehat{\phi}(\oz),
\end{align*}
where
\begin{align*}
r_n:=\int_G f(x)\overline{\phi(\sz(x)\mz n)}\,d\mu(x)\quad{\rm and}\quad
h(\oz):=\sum_{n\in\Lz}r_n\overline{\chi(n,\oz)}
\end{align*}
is a $\Lz$-periodic function which belongs to $L^2(U^*)$. This finishes the proof of (ii)
and hence Lemma \ref{4l2}.
\end{proof}

The succeeding theorem is the second main result of this article, which
provides a different method from that used in \cite{Far08} for the construction of wavelets.

\begin{theorem}\label{t2}
Let $G$ be a Vilenkin group and $\Oz=\{\Oz^{(u)}\}_{u\in\Pz}\subset G^*$ a wavelet set.
For any given $u\in\Pz$, let
$$\widehat{\psi^{(u)}}:=\mathbf{1}_{\Oz^{(u)}}.$$
Then $\{\psi^{(u)}\}_{u\in\Pz}$ is associated with
an {\rm MRA} if and only if
\begin{align}\label{4e4}
\mu^*\lf(\Oz_\Sigma\cap\lf(\Oz_\Sigma\pz n\r)\r)=
\begin{cases}
1
&if\ n=\theta,\\
0
&if\ n\in\Lz\backslash\{\theta\},
\end{cases}
\end{align}
where $\theta:=(\ldots,0,0,\ldots)$ denotes the identity element of $G^*$ and
$$\Oz_\Sigma:=\bigcup_{u\in\Pz}\bigcup_{j\in\nn}\sz^j\lf(\Oz^{(u)}\r).$$
\end{theorem}

\begin{proof}
We prove the present theorem by two steps.

\emph{Step 1}. In this step, we show the necessity. For this purpose, by the assumption $\Oz=\{\Oz^{(u)}\}_{u\in\Pz}$ is a wavelet set and Theorem \ref{t1}, we find that the sets in the union
of $\Oz_\Sigma$ are almost mutually disjoint, that is,
for any $u_1,u_2\in\Pz$ and $j,k\in\nn$ with $u_1\neq u_2$ or $j\neq k$,
\begin{align}\label{4e5}
\mu^*\lf(\sz^j\lf(\Oz^{(u_1)}\r)\cap\sz^k\lf(\Oz^{(u_2)}\r)\r)=0
\end{align}
and hence
\begin{align}\label{4e6}
\mu^*\lf(\Oz_\Sigma\r)&=\mu^*\lf(\bigcup_{u\in\Pz}\bigcup_{j\in\nn}\sz^j\lf(\Oz^{(u)}\r)\r)\\
&=\sum_{u\in\Pz}\sum_{j\in\nn}\mu^*\lf(\sz^j\lf(\Oz^{(u)}\r)\r)=1,\noz
\end{align}
where we used \eqref{3e1} to obtain the last equal relation.

Assume now $\{\psi^{(u)}\}_{u\in\Pz}$ is associated with an MRA. Then, from Lemma \ref{4l1}, we deduce that
there exists some scaling function $\phi$ such that, for almost every $\oz\in G^*$,
\begin{align*}
\lf|\widehat{\phi}(\oz)\r|^2
=\sum_{u\in\Pz}\sum_{j\in\nn}\lf|\widehat{\psi^{(u)}}\lf(\rho^j(\oz)\r)\r|^2.
\end{align*}
This, together with the definitions of both $\{\psi^{(u)}\}_{u\in\Pz}$ and $\Oz_\Sigma$
as well as \eqref{4e5}, implies that
\begin{align}\label{4e7}
\widehat{\phi}(\oz)=\mathbf{1}_{\Oz_\Sigma}(\oz)\quad\rm{for\ almost\ every}\ \oz\in G^*.
\end{align}
Observe that $\{\phi(\cdot\pz n)\}_{n\in\Lz}$ is an orthonormal set. From Lemma \ref{2l1},
it follows that, for almost every $\oz\in G^*$,
\begin{align}\label{4e8}
\sum_{n\in\Lz}\lf|\widehat{\phi}(\oz\pz n)\r|^2=1.
\end{align}
Combining this, \eqref{4e6} and \eqref{4e7}, we conclude that \eqref{4e4} holds true. Thus,
the proof of the necessity of the present theorem is completed.

\emph{Step 2}. In this step, we prove the sufficiency. To this end, suppose that $\Oz_\Sigma$
satisfies \eqref{4e4}.
For almost every $\oz\in G^*$, let
\begin{align}\label{4e8'}
\widehat{\phi}(\oz):=\mathbf{1}_{\Oz_\Sigma}(\oz).
\end{align}
Then, by \eqref{4e4} and \eqref{4e6}, we know that \eqref{4e8} is true and hence
$\{\phi(\cdot\mz n)\}_{n\in\Lz}$ forms an orthonormal system on account of Lemma \ref{2l1}.

For any given $u\in\Pz$ and any $\oz\in \Oz_\Sigma$, let
\begin{align*}
m_0(\oz):=
\begin{cases}
\vspace{0.2cm}
0
&\text{if}\ \oz\in\displaystyle\bigcup_{u\in\Pz}\sz\lf(\Oz^{(u)}\r),\\
1
&\text{if}\ \oz\in\displaystyle\bigcup_{u\in\Pz}\bigcup_{j\in[2,\fz)\cap\nn}\sz^j\lf(\Oz^{(u)}\r)
\end{cases}
\end{align*}
and
\begin{align*}
m_1^{(u)}(\oz):=
\begin{cases}
\vspace{0.2cm}
1
&\text{if}\ \oz\in\sz\lf(\Oz^{(u)}\r),\\
0
&\text{if}\ \oz\in\Oz_\Sigma\backslash\sz\lf(\Oz^{(u)}\r).
\end{cases}
\end{align*}
In addition, from Theorem \ref{t1}, we infer that $\{\Oz_\Sigma\pz n:\ n\in\Lz\}$ partitions $G^*$.
Therefore, every element of
$\{m_0,m_1^{(u)}:\ u\in\Pz\}$ can be uniquely extended to a $\Lz$-periodic function on $G^*$.

For any given $u\in\Pz$ and $\oz\in G^*$, let $\widehat{\psi^{(u)}}(\oz):=\mathbf{1}_{\Oz^{(u)}}(\oz)$.
Next, we show that the set $\{\psi^{(u)}\}_{u\in\Pz}$ is associated with
an MRA by two substeps.

\emph{Substep 2.1}. In this substep, we prove that, for any given $u\in\Pz$ and almost every $\oz\in G^*$,
the following two equalities hold true:
\begin{align}\label{4e9}
\widehat{\phi}(\rho(\oz))=m_0(\oz)\widehat{\phi}(\oz)
\end{align}
and
\begin{align}\label{4e10}
\widehat{\psi^{(u)}}(\rho(\oz))=m_1^{(u)}(\oz)\widehat{\phi}(\oz).
\end{align}

For \eqref{4e10}, if $\oz\in\sz(\Oz^{(u)})$ with $u\in\Pz$, then $m_1^{(u)}(\oz)=1$. By this, the definitions
of both $\widehat{\psi^{(u)}}$ and $\widehat{\phi}$, we find that, for almost every $\oz\in\sz(\Oz^{(u)})$,
\begin{align}\label{4e11}
\widehat{\psi^{(u)}}(\rho(\oz))=\mathbf{1}_{\Oz^{(u)}}(\rho(\oz))
=\mathbf{1}_{\sz(\Oz^{(u)})}(\oz)=1=m_1^{(u)}(\oz)\widehat{\phi}(\oz).
\end{align}
If $\oz\in G^*\backslash\sz(\Oz^{(u)})$, then similarly to \eqref{4e11},
we have
\begin{align*}
\widehat{\psi^{(u)}}(\rho(\oz))=0=
\begin{cases}
\vspace{0.2cm}
0\cdot\widehat{\phi}(\oz)=m_1^{(u)}(\oz)\widehat{\phi}(\oz)
&\text{for\ almost\ every}\ \oz\in\Oz_\Sigma\backslash\sz\lf(\Oz^{(u)}\r),\\
m_1^{(u)}(\oz)\cdot 0=m_1^{(u)}(\oz)\widehat{\phi}(\oz)
&\text{for\ almost\ every}\ \oz\in G^*\backslash\Oz_\Sigma.
\end{cases}
\end{align*}
This, combined with \eqref{4e11}, finishes the proof of \eqref{4e10}.

For \eqref{4e9}, it is sufficient to show that, for almost every $\oz\in G^*$,
\begin{align}\label{4e12}
\widehat{\phi}(\oz)=\prod_{j\in\nn}m_0\lf(\sz^j(\oz)\r).
\end{align}
Indeed, it follows from \eqref{4e12} that, for almost every $\oz\in G^*$,
\begin{align*}
\widehat{\phi}(\rho(\oz))&=\prod_{j\in\nn}m_0\lf(\sz^{j-1}(\oz)\r)\\
&=m_0(\oz)\prod_{j\in\nn}m_0\lf(\sz^j(\oz)\r)=m_0(\oz)\widehat{\phi}(\oz).
\end{align*}
This is just \eqref{4e9}.

Next, we prove \eqref{4e12}. Note that, for almost every $\oz\in G^*$,
there exists a unique integer $j$ such that
$$\oz\in\bigcup_{u\in\Pz}\rho^j\lf(\Oz^{(u)}\r)=\bigcup_{u\in\Pz}\sz^{-j}\lf(\Oz^{(u)}\r).$$
If $j\in\zz_+$,
then $\oz\in G^*\backslash\Oz_\Sigma$ and hence $\widehat{\phi}(\oz)=0$. Moreover,
\begin{align*}
\oz\in\bigcup_{u\in\Pz}\rho^j\lf(\Oz^{(u)}\r)
&\Longrightarrow\sz^j(\oz)\in\bigcup_{u\in\Pz}\Oz^{(u)}
\Longrightarrow\sz^{j+1}(\oz)\in\bigcup_{u\in\Pz}\sz\lf(\Oz^{(u)}\r)\\
&\Longrightarrow m_0\lf(\sz^{j+1}(\oz)\r)=0
\Longrightarrow\prod_{j\in\nn}m_0\lf(\sz^j(\oz)\r)=0,
\end{align*}
where we used $j+1\in\nn$ to obtain the last `$\Longrightarrow$'. If $j\in\zz\backslash\zz_+$,
then $-j\in\nn$ and hence $\oz\in \Oz_\Sigma$. Thus, $\widehat{\phi}(\oz)=1$. Furthermore,
for any $k\in\nn$, $(k-j)\in[2,\fz)\cap\nn$. Then we have
\begin{align*}
\oz\in\bigcup_{u\in\Pz}\rho^j\lf(\Oz^{(u)}\r)
&\Longrightarrow\sz^k(\oz)\in\bigcup_{u\in\Pz}\sz^{k-j}\lf(\Oz^{(u)}\r),\ \forall\,k\in\nn\\
&\Longrightarrow m_0\lf(\sz^k(\oz)\r)=1,\ \forall\,k\in\nn
\Longrightarrow\prod_{k\in\nn}m_0\lf(\sz^k(\oz)\r)=1.
\end{align*}
Both cases give us
\begin{align*}
\widehat{\phi}(\oz)=\prod_{j\in\nn}m_0\lf(\sz^j(\oz)\r),
\end{align*}
which completes the proof of \eqref{4e12} and hence of \eqref{4e9}.

\emph{Substep 2.2}. In this substep, we show that the function $\phi$ defined by \eqref{4e8'} is a scaling
function for the set
\begin{align}\label{4e13}
V_0:=\bigoplus_{\ell\in\zz\backslash\zz_+}W_\ell
:=\bigoplus_{\ell\in\zz\backslash\zz_+}\overline{\text{span}\lf\{\psi_{\ell,n}^{(u)}:\ n\in\Lz,u\in\Pz\r\}}.
\end{align}
For this purpose, we only need to prove that
\begin{align}\label{4e14}
V_0=\widetilde{V}_0
:=\overline{\text{span}\lf\{\phi(\cdot\mz n):\ n\in\Lz\r\}}.
\end{align}

From \eqref{3e9'}, \eqref{4e10}, \eqref{4e9}, we deduce that, for any $u\in\Pz$, $j\in[2,\fz)\cap\nn$ and
almost every $\oz\in G^*$,
\begin{align}\label{4e16}
\lf(\psi^{(u)}_{-j,0}\r)^\wedge(\oz)&=p^{j/2}\widehat{\psi^{(u)}}\lf(\rho^j(\oz)\r)\\
&=p^{j/2}m_1^{(u)}\lf(\rho^{j-1}(\oz)\r)\widehat{\phi}\lf(\rho^{j-1}(\oz)\r)\noz\\
&=p^{j/2}m_1^{(u)}\lf(\rho^{j-1}(\oz)\r)m_0\lf(\rho^{j-2}(\oz)\r)\widehat{\phi}\lf(\rho^{j-2}(\oz)\r)\noz\\
&=p^{j/2}m_1^{(u)}\lf(\rho^{j-1}(\oz)\r)
\prod_{\ell\in[0,j-2]\cap\zz_+}m_0\lf(\rho^{\ell}(\oz)\r)\widehat{\phi}(\oz)\noz\\
&=m_j^{(u)}(\oz)\widehat{\phi}(\oz),\noz
\end{align}
where
\begin{align*}
m_j^{(u)}(\oz)
:=p^{j/2}m_1^{(u)}\lf(\rho^{j-1}(\oz)\r)\prod_{\ell\in[0,j-2]\cap\zz_+}m_0\lf(\rho^{\ell}(\oz)\r).
\end{align*}
Observe that each element of $\{m_0,m_1^{(u)}:\ u\in\Pz\}$ is a $\Lz$-periodic function with an upper bound
of 1, and the fact that $\rho(n)\in\Lz$ for any $n\in\Lz$. We immediately obtain that, for any $u\in\Pz$
and $j\in[2,\fz)\cap\nn$, $m_j^{(u)}\in L^2(U^*)$ and
\begin{align*}
m_j^{(u)}(\oz\pz n)&=p^{j/2}m_1^{(u)}\lf(\rho^{j-1}(\oz)\pz\rho^{j-1}(n)\r)
\prod_{\ell\in[0,j-2]\cap\zz_+}m_0\lf(\rho^{\ell}(\oz)\pz\rho^{\ell}(n)\r)\\
&=p^{j/2}m_1^{(u)}\lf(\rho^{j-1}(\oz)\r)\prod_{\ell\in[0,j-2]\cap\zz_+}m_0\lf(\rho^{\ell}(\oz)\r)\\
&=m_j^{(u)}(\oz)\quad\text{for\ almost\ every}\ \oz\in G^*.
\end{align*}
Therefore, every $m_j^{(u)}$ is a $\Lz$-periodic $L^2(U^*)$-function. This, together with \eqref{4e16}
and Lemma \ref{4l2}(i), implies that, for any $u\in\Pz$ and $j\in[2,\fz)\cap\nn$,
$\psi^{(u)}_{-j,0}\in\widetilde{V}_0$. In addition, by \eqref{4e10} and and Lemma \ref{4l2}(i) again,
it is easy to see that, for almost every $\oz\in G^*$,
\begin{align*}
\lf(\psi^{(u)}_{-1,0}\r)^\wedge(\oz)=p^{1/2}\widehat{\psi^{(u)}}\lf(\rho(\oz)\r)
=p^{1/2}m_1^{(u)}(\oz)\widehat{\phi}(\oz).
\end{align*}
Thus,
$$\lf\{\psi_{\ell,0}^{(u)}:\ \ell\in\zz\backslash\zz_+,u\in\Pz\r\}\subset \widetilde{V}_0,$$
which, combined with the definition of $\widetilde{V}_0$ [see \eqref{4e14}], further implies that
$$\lf\{\psi_{\ell,n}^{(u)}:\ \ell\in\zz\backslash\zz_+,n\in\Lz,u\in\Pz\r\}\subset\widetilde{V}_0$$
since $\widetilde{V}_0$ is invariant under $\Lz$-translation. By this and \eqref{4e13}, we conclude that
\begin{align*}
V_0=\bigoplus_{\ell\in\zz\backslash\zz_+}W_\ell\subset\widetilde{V}_0.
\end{align*}

We next show the converse inclusion, that is, $\widetilde{V}_0\subset V_0.$ To this end,
by the almost mutual disjointness of
$\{\sz(\Oz^{(u)})\}_{u\in\Pz}$ and the definitions of $\{m_0,m_1^{(u)}:\ u\in\Pz\}$, we know that,
for any $u,\tau\in\Pz$ with $u\neq\tau$ and almost every $\oz\in G^*$,
\begin{align}\label{4e17}
m_0(\oz)\overline{m^{(u)}_1(\oz)}+\sum_{x\in\Pz}m_0(\oz\pz0.x)\overline{m^{(u)}_1(\oz\pz0.x)}=0
\end{align}
and
\begin{align*}
m^{(u)}_1(\oz)\overline{m^{(\tau)}_1(\oz)}+
\sum_{x\in\Pz}m^{(u)}_1(\oz\pz0.x)\overline{m^{(\tau)}_1(\oz\pz0.x)}=0,
\end{align*}
where $0.x:=(\nu_j)_{j\in\zz}$ is the element of $G^*$ satisfying that $\nu_1=x\in\Pz$
and $\nu_j=0$ for any $j\in\zz\setminus\{1\}$.

Moreover, by \eqref{4e9} and the $\Lz$-periodicity of $m_0$, we find that, for almost every $\oz\in G^*$
and any $n\in\Lz$,
\begin{align*}
m_0(\oz)\widehat{\phi}(\oz\pz n)=m_0(\oz\pz n)\widehat{\phi}(\oz\pz n)=\widehat{\phi}(\rho(\oz\pz n)).
\end{align*}
From this, the mutual orthogonality of $\{\phi(\cdot\mz n)\}_{n\in\Lz}$ and Lemma \ref{2l2},
we infer that, for almost every $\oz\in G^*$,
\begin{align}\label{4e18}
&\sum_{n\in\Lz}\lf|m_0(\oz)\r|^2\lf|\widehat{\phi}(\oz\pz n)\r|^2\\
&\quad=\sum_{n\in\Lz}\lf|\widehat{\phi}(\rho(\oz\pz n))\r|^2
=\sum_{n\in\Lz_1}\lf|\widehat{\phi}(\rho\oz\pz n)\r|^2\noz\\
&\quad=\sum_{n\in\Lz}\lf|\widehat{\phi}(\rho\oz\pz n)\r|^2
-\sum_{x\in\Pz}\sum_{n\in\Lz_1}\lf|\widehat{\phi}(\rho\oz\pz n\pz x.0)\r|^2\noz\\
&\quad=1-\sum_{x\in\Pz}\sum_{n\in\Lz_1}\lf|\widehat{\phi}(\rho(\oz\pz \sz(n)\pz 0.x))\r|^2\noz\\
&\quad=1-\sum_{x\in\Pz}\sum_{n\in\Lz}\lf|\widehat{\phi}(\rho(\oz\pz n\pz 0.x))\r|^2\noz\\
&\quad=1-\sum_{x\in\Pz}\sum_{n\in\Lz}\lf|\widehat{\phi}(\oz\pz n\pz 0.x)\r|^2\lf|m_0(\oz+0.x)\r|^2\noz\\
&\quad=1-\sum_{x\in\Pz}\lf|m_0(\oz+0.x)\r|^2,\noz
\end{align}
where $0.x$ is as in \eqref{4e17} and $x.0:=(\nu_j)_{j\in\zz}$ is the element of $G^*$ satisfying
that $\nu_0=x\in\Pz$ and $\nu_j=0$ for any $j\in\zz\setminus\{0\}$. Therefore, for almost every $\oz\in G^*$,
\begin{align*}
\lf|m_0(\oz)\r|^2+\sum_{x\in\Pz}\lf|m_0(\oz+0.x)\r|^2=1.
\end{align*}
In addition, using the Parseval equality (see, for instance, \cite[Proposition 1(c)]{Far07}), \eqref{3e2'},
we obtain that, for any given $u\in\Pz$ and any $m,r\in\Lz$,
\begin{align*}
&\lf(\psi^{(u)}(\cdot\mz m),\psi^{(u)}(\cdot\mz r)\r)\\
&\quad=\lf(\lf(\psi^{(u)}(\cdot\mz m)\r)^\wedge,\lf(\psi^{(u)}(\cdot\mz r)\r)^\wedge\r)\\
&\quad=\int_{G^*}\lf(\psi^{(u)}(\cdot\mz m)\r)^\wedge(\oz)
\overline{\lf(\psi^{(u)}(\cdot\mz r)\r)^\wedge(\oz)}\,d\mu^*(\oz)\\
&\quad=\int_{G^*}\overline{\chi(m,\oz)}\lf(\psi^{(u)}\r)^\wedge(\oz)
\chi(r,\oz)\overline{\lf(\psi^{(u)}\r)^\wedge(\oz)}\,d\mu^*(\oz)\\
&\quad=\int_{\Oz^{(u)}}\overline{\chi(m,\oz)}\chi(r,\oz)\,d\mu^*(\oz).
\end{align*}
By this, Theorem \ref{t1}(ii) and the mutual orthogonality of
$\{\chi(r,\cdot)\}_{r\in\Lz}$ on $U^*$ (see, for instance, \cite[p.\,311]{Far08}), it is easy to see that,
for any given $u\in\Pz$ and any $m,r\in\Lz$,
\begin{align*}
\lf(\psi^{(u)}(\cdot\mz m),\psi^{(u)}(\cdot\mz r)\r)
&=\sum_{n\in N^{(u)}}\int_{\Oz_n^{(u)}}\overline{\chi\lf(m,\oz\r)}\chi\lf(r,\oz\r)\,d\mu^*(\oz)\\
&=\sum_{n\in N^{(u)}}\int_{\Oz_n^{(u)}\mz n}\overline{\chi\lf(m,\oz\r)}\chi\lf(r,\oz\r)\,d\mu^*(\oz)\\
&=\int_{U^*}\overline{\chi\lf(m,\oz\r)}\chi\lf(r,\oz\r)\,d\mu^*(\oz)\\
&=\begin{cases}
1
&\text{if}\ m=r,\\
0
&\text{if}\ m\neq r.
\end{cases}
\end{align*}
Thus, for any given $u\in\Pz$, $\{\psi^{(u)}(\cdot\mz n)\}_{n\in\Lz}$ forms an orthonormal set, which
together with Lemma \ref{2l2}, further implies that, for almost every $\oz\in G^*$,
\begin{align*}
\sum_{n\in\Lz}\lf|\widehat{\psi^{(u)}}(\oz\pz n)\r|^2=1.
\end{align*}
Via this and a calculation parallel to \eqref{4e18}, we conclude that, for any given $u\in\Pz$ and
almost every $\oz\in G^*$,
\begin{align*}
\lf|m^{(u)}_1(\oz)\r|^2+\sum_{x\in\Pz}\lf|m^{(u)}_1\lf(\oz\pz0.x\r)\r|^2=1.
\end{align*}

On the other hand, from \eqref{4e9}, \eqref{4e10}, the $\Lz$-periodicity of $\{m_0,m_1^{(u)}:\ u\in\Pz\}$
and \eqref{4e17}, it follows that, for any given $u\in\Pz$ and almost every $\oz\in G^*$,
\begin{align*}
&\sum_{n\in\Lz}\widehat{\phi}(\oz\pz n)\overline{\widehat{\psi^{(u)}}(\oz\pz n)}\\
&\quad=\sum_{n\in\Lz}m_0(\sz(\oz\pz n))\overline{m^{(u)}_1(\sz(\oz\pz n))}
\lf|\widehat{\phi}(\sz(\oz\pz n))\r|^2\\
&\quad=\sum_{n\in\Lz_1}m_0(\sz(\oz\pz n))\overline{m^{(u)}_1(\sz(\oz\pz n))}
\lf|\widehat{\phi}(\sz(\oz\pz n))\r|^2\\
&\qquad+\sum_{x\in\Pz}\sum_{n\in\Lz_1}m_0(\sz(\oz\pz n\pz x.0))\overline{m^{(u)}_1(\sz(\oz\pz n \pz x.0))}
\lf|\widehat{\phi}(\sz(\oz\pz n \pz x.0))\r|^2\\
&\quad=\sum_{n\in\Lz}m_0(\sz(\oz))\overline{m^{(u)}_1(\sz(\oz))}
\lf|\widehat{\phi}(\sz(\oz)\pz n)\r|^2\\
&\qquad+\sum_{x\in\Pz}\sum_{n\in\Lz}m_0(\sz(\oz)\pz 0.x)\overline{m^{(u)}_1(\sz(\oz)\pz 0.x)}
\lf|\widehat{\phi}(\sz(\oz\pz x.0) \pz n)\r|^2\\
&\quad=m_0(\sz(\oz))\overline{m^{(u)}_1(\sz(\oz))}
+\sum_{x\in\Pz}m_0(\sz(\oz)\pz 0.x)\overline{m^{(u)}_1(\sz(\oz)\pz 0.x)}\\
&\quad=0,
\end{align*}
where, for any $x\in\Pz$, $0.x$ and $x.0$ are as in \eqref{4e18}. By this, \eqref{4e9} and
the $\Lz$-periodicity of $m_0$ again,
we find that, for any given $u\in\Pz$ and for any $j\in\zz_+$ and almost every $\oz\in G^*$,
\begin{align*}
\sum_{n\in\Lz}\widehat{\phi}\lf(\rho^j(\oz\pz n)\r)\overline{\widehat{\psi^{(u)}}(\oz\pz n)}
&=\sum_{n\in\Lz}\prod_{\ell\in[0,j-1]\cap\nn}m_0\lf(\rho^\ell(\oz\pz n)\r)
\widehat{\phi}(\oz\pz n)\overline{\widehat{\psi^{(u)}}(\oz\pz n)}\\
&=\prod_{\ell\in[0,j-1]\cap\nn}m_0\lf(\rho^\ell(\oz)\r)
\sum_{n\in\Lz}\widehat{\phi}(\oz\pz n)\overline{\widehat{\psi^{(u)}}(\oz\pz n)}\\
&=0.
\end{align*}
Combining this and the Parseval equality, we obtain that, for any given $u\in\Pz$ and for any $j\in\zz_+$
and $k\in\Lz$,
\begin{align*}
\lf(\phi,\psi^{(u)}_{j,k}\r)
&=\lf(\widehat{\phi},\widehat{\psi^{(u)}_{j,k}}\r)\\
&=\int_{G^*}\widehat{\phi}(\oz)\overline{\widehat{\psi^{(u)}_{j,k}}(\oz)}\,d\mu^*(\oz)\\
&=\int_{G^*}\widehat{\phi}(\oz)\overline{\int_Gp^{j/2}
\psi^{(u)}\lf(\rho^j(x)\mz k\r)\overline{\chi(x,\oz)}\,dx}\,d\mu^*(\oz)\\
&=\int_{G^*}\widehat{\phi}(\oz)\overline{\int_Gp^{j/2}p^{-j}
\psi^{(u)}\lf(z\r)\overline{\chi\lf(\sz^j(z\pz k),\oz\r)}\,dz}\,d\mu^*(\oz)\\
&=p^{-j/2}\int_{G^*}\widehat{\phi}(\oz)\overline{\int_G\psi^{(u)}(z)
\overline{\chi\lf(z,\sz^j(\oz)\r)}\,\overline{\chi\lf(k,\sz^j(\oz)\r)}\,dz}\,d\mu^*(\oz)\\
&=p^{-j/2}\int_{G^*}\widehat{\phi}(\oz)\overline{\widehat{\psi^{(u)}}\lf(\sz^j(\oz)\r)}
\chi\lf(k,\sz^j(\oz)\r)\,d\mu^*(\oz)\\
&=p^{j/2}\int_{G^*}\widehat{\phi}\lf(\rho^j(\oz)\r)\overline{\widehat{\psi^{(u)}}(\oz)}
\chi(k,\oz)\,d\mu^*(\oz)\\
&=p^{j/2}\int_{U^*}\sum_{n\in\Lz}\widehat{\phi}\lf(\rho^j(\oz\pz n)\r)
\overline{\widehat{\psi^{(u)}}(\oz\pz n)}\chi(k,\oz\pz n)\,d\mu^*(\oz)\\
&=0.
\end{align*}
Therefore, for any $j\in\zz_+$,
$$\phi\perp W_j:=\overline{\text{span}\lf\{\psi_{j,n}^{(u)}:\ n\in\Lz,u\in\Pz\r\}}.$$
Observe that each $W_j$ is invariant under $\Lz$-translation. Then we have
\begin{align*}
\overline{\text{span}\lf\{\phi(\cdot\mz n):\ n\in\Lz\r\}}\subset
\bigoplus_{j\in\zz\backslash\zz_+}\overline{\text{span}\lf\{\psi_{j,n}^{(u)}:\ n\in\Lz,u\in\Pz\r\}},
\end{align*}
which gives $\widetilde{V}_0\subset V_0$. This proves \eqref{4e14} and hence $\phi$ is a scaling
function for $V_0$ defined in \eqref{4e13}.

Obviously, for any $u\in\Pz$, $\psi^{(u)}$ is obtained from the scaling function $\phi$ and the
scaling filter $m_0$ and wavelet filter $m_1^{(u)}$, that is, $\psi^{(u)}$ is associated with an MRA.
This finishes the proof of Theorem \ref{t2}.
\end{proof}

\begin{remark}\label{4r}
When $p=2$, the Vilenkin group $G$ used in present article goes back to the Cantor dyadic group
(see, for instance, \cite{Lang96,ms21}). Thus,
Theorem \ref{t2} includes \cite[Theorem 6]{ms21} as a special case.
\end{remark}

\medskip

\noindent  Jun Liu and Chi Zhang (Corresponding author)

\smallskip

\noindent School of Mathematics, JCAM,
China University of Mining and Technology,
Xuzhou 221116, The People's Republic of China

\smallskip

\noindent {\it E-mails}: \texttt{junliu@cumt.edu.cn} (J. Liu)

\noindent\phantom{{\it E-mails:} }\texttt{zclqq32@cumt.edu.cn} (C. Zhang)

\end{document}